\documentclass[a4paper]{amsart}
\usepackage{amssymb,enumerate}
\RequirePackage{pgf,pgffor}
\usepackage{longtable} 

\usepackage{hhline} 
\usepackage{multicol}

\newtheorem{thm}{Theorem}[section]
\newtheorem{thm*}{Theorem}[]
\newtheorem{lem}[thm]{Lemma}

\newtheorem{prop}[thm]{Proposition}

\newtheorem{conj}[thm]{Conjecture}

\theoremstyle{definition}

\newtheorem{rem}[thm]{Remark}
\newtheorem{defi}[thm]{Definition}

\newtheorem{exa}[thm]{Example}
\newtheorem{Notation}[thm]{Notation}

\newcommand{\Irr}{\mathrm{Irr}}

\newcommand{\Res}{\mathrm{Res}}
\newcommand{\bG} {\mathbf G}
\newcommand{\bL} {\mathbf L}
\newcommand{\bH} {\mathbf H}

\newcommand{\bM} {\mathbf M}
\newcommand{\bT} {\mathbf T}

\newcommand{\bV} {\mathbf V}

\newcommand{\bN} {\mathbf N}
\newcommand{\cE} {\mathcal E}

\newcommand{\cO} {\mathcal O}

\newcommand{\la}{\la}

\let\la=\lambda

\newcommand{\FF}{{\mathbb{F}}}

\begin{document}

\title{Rationality of blocks of quasi-simple finite groups}
\date{\today}
\author{Niamh Farrell}
\address{Fachbereich Mathematik, TU Kaiserslautern, Postfach 3049, 67653 Kaiserslautern, Germany}
\email{farrell@mathematik.uni-kl.de}
\author{Radha Kessar}
\address{Department of Mathematics, City, University of London,
  Northampton Square, EC1V 0HB London UK}
\email{radha.kessar.1@city.ac.uk}

\begin{abstract}
Let $\ell $ be a prime number. We show that the  Morita Frobenius number of   an   $\ell$-block of a  quasi-simple finite group   is   at  most $4$  and that  the  strong  Frobenius number  is  at most   $4 |D|^2!$, where $D$ denotes a defect group of the block.  We deduce that  a  basic  algebra of any block  of   the group algebra  of a  quasi-simple finite  group over an algebraically closed field of characteristic $\ell$     is  defined over  a  field with  $\ell^a $  elements   for some  $ a \leq 4 $. We derive   consequences for Donovan's conjecture. In particular, we show that Donovan's conjecture holds for  $\ell$-blocks of   special linear groups.
\end{abstract}

\thanks{This article was  partly written while  the  authors were visiting the Mathematical Sciences Research Institute in Berkeley, 
California in Spring 2018 for  the  programme Group Representation Theory and Applications supported by the National Science Foundation under Grant No. DMS-1440140. It is a pleasure to thank the institute for its hospitality. \\
The first author also gratefully acknowledges financial support from the DFG project SFB-TRR 195} 

\maketitle

\section{Introduction and Main Results}
Let $\ell$ be a prime number  and  let $(K, \cO, k)$ be an $\ell$-modular system with $k \cong \overline \FF_\ell$  such that  $\cO  $  is absolutely  unramified. The following is  the  main  result of this paper (see Section~\ref{sec:prelims} for  notation and  definitions).

\begin{thm}\label{thm:newmaintheorem} 
Let $G$ be a quasi-simple finite group and let $B$ be a block  algebra of $\cO G$.  Let  $D$ be a defect group of $B$.
\begin{enumerate}[\rm(i)] 
	\item The Morita Frobenius numbers of $B$ and  of   $k\otimes_{\cO}B$ are at most $4$. 
	\item  The strong Frobenius number of $B$ is at most  $4|D|^2!$.
\end{enumerate} 
\end{thm}

In many cases  in  the above theorem  we show that the  Morita Frobenius number  of $B$ is  in fact equal  to $1$ (see Theorem~\ref{thm:overallresult}),  and in no case do we show that the number is greater than $1$. We note that there are examples  of  blocks of $\ell$-solvable groups  with Morita Frobenius number equal to $2$ \cite{B/K}.    In \cite{F},  the  first author  calculated the Morita Frobenius  numbers of   $k \otimes_{\cO} B$    for  several  families of blocks  $B$   of  quasi-simple finite groups.   The main  outstanding  case,   which we treat   in the current paper,   was that  of non-unipotent blocks   of   quasi-simple finite  groups of Lie type  in non-describing characteristic.

The  motivation for  Theorem~\ref{thm:newmaintheorem}   comes from  Donovan's conjecture  whose statement we  recall. 

\begin{conj}[Donovan's Conjecture {\cite[Conjecture M]{Apaper}}]
Let $D$ be a finite $\ell$-group. There are finitely many Morita equivalence classes of blocks of finite group algebras over  $k$ with defect groups isomorphic to $D$.
\end{conj}

A   weaker  version of Donovan's   conjecture states that the entries of the Cartan matrices of blocks with defect groups isomorphic to $D$ are bounded by some function which depends only on $|D|$.  The gap between  the weak and strong forms is precisely  a rationality question. This was first   observed  by Hiss \cite{H2}  and the theme was  developed in   \cite[Theorem 1.4]{K}. 

\begin{conj}[Rationality Conjecture {\cite[Conjecture 1.3]{K}}]\label{conj:rat} The Morita Frobenius numbers of blocks  of finite  group algebras  over $k$  with defect groups isomorphic to $D$ are bounded by a function which depends only on $|D|$.\end{conj}

All three  conjectures are open  at present.  The first part  of Theorem~\ref{thm:newmaintheorem} shows that the Rationality conjecture holds  for blocks of  quasi-simple  finite  groups.  

In \cite[Theorem 8.6]{H/K2}  it was shown that blocks of   finite special linear groups   satisfy  the weak Donovan conjecture.   Combining this result with the first part of Theorem~\ref{thm:newmaintheorem}  we obtain Donovan's conjecture    for  blocks of finite special linear groups. This result first appeared in {\cite[Theorem C]{Fthesis}}.

\begin{thm}  \label{thm:speciallinear} Let $D$ be a finite $\ell$-group. There are finitely many Morita equivalence classes of blocks of group algebras over $k$ of finite special linear groups with defect groups isomorphic to $D$.
\end{thm} 

Motivation for the second part of Theorem~\ref{thm:newmaintheorem}  comes from a recent reduction of Donovan's conjecture for blocks with abelian defect, proved by Eaton and Livesey \cite{E/L}. 

\begin{thm} \label{thm:ELred} \cite[Theorem 1.4]{E/L}
Suppose that there exist functions $\epsilon, \gamma : \mathbb{N} \rightarrow \mathbb{N}$ such that for all block algebras $B$ of group algebras over $\cO $ of quasi-simple finite groups with abelian defect groups of order $\ell^d$, the strong Frobenius number of $B$ is bounded by $\epsilon (d)$ and all Cartan invariants of $ k \otimes_{\cO} B$  are at most $\gamma (d)$. Then Donovan's conjecture holds for   blocks with abelian defect groups.
\end{thm}

The second part of Theorem~\ref{thm:newmaintheorem}  in combination  with the  above reduction result  shows that  in order to prove Donovan's conjecture for  blocks with abelian defect groups, it  remains to show that the weak Donovan conjecture  holds for blocks of quasi-simple finite groups with abelian defect groups.  Note that unlike     Theorem \ref{thm:ELred}, the hypothesis of  \cite[Theorem 1.4]{E/L}    
does not  stipulate that $\cO$ is absolutely unramified.  We also note that   Eaton,  Eisele  and Livesey recently  strengthened the conclusion   of Theorem \ref{thm:ELred}  to  apply to  the  version of  Donovan's conjecture   dealing with Morita equivalence  classes of blocks  of finite group algebras over   ${\mathcal  \cO}$ \cite{E/E/L}.

Recall that a finite dimensional $k$-algebra $A$ is said to be defined over a subfield $F$ of $k$ if there exists an $F$-algebra $A_0$ such that $A \cong k \otimes_{F} A_0$. The first part of Theorem~\ref{thm:newmaintheorem} may be recast as the following rationality statement.

\begin{thm}\label{thm:sideshow}
Let $A$ be a basic algebra of a block algebra of $kG$ where $G$ is a quasi-simple finite group. Then $A$ is defined over $\mathbb{F}_{\ell^a}$  for some $a \leq 4 $. 
\end{thm}

We begin in Section~\ref{sec:prelims} by recalling some preliminaries including background results on  rationality  and covering and dominating blocks. Section~\ref{sec:alggroups}  studies the effect of  $\ell$-adic  field automorphisms on characters  of finite reductive groups in  characteristic different from $\ell$. In Section~\ref{sec:BDR} we refine the results of the previous section  to the setting of the key reduction theorem of Bonnaf\'{e}-Dat-Rouquier \cite[Section 7]{B/D/R}. Theorems~\ref{prop:BDRresult} and \ref{prop:central} are then used to prove almost all of the results for  blocks of the finite groups of Lie type in non-defining characteristic  in  Section~\ref{sec:FGLT}; the  remaining  cases  in  type $E_8$ are dealt with on  an ad hoc basis  in Section~\ref{subsec:E8}. Sections~\ref{sec:reesuz}  and   \ref{sec:analogues}  cover  the case of  defining characteristic,  Ree and Suzuki groups, alternating and sporadic groups and exceptional covering groups. The proofs of  Theorem~\ref{thm:newmaintheorem},   Theorem~\ref{thm:speciallinear} and Theorem~\ref{thm:sideshow}  are given  in Section~\ref{sec:proof}.

\bigskip

\textbf{Acknowledgements.} We thank  Marc Cabanes and Fran\c{c}ois Digne for their help   in clearing up  a  point in the  proof of  Lemma~\ref{lem:duality}. We  also  thank  Gunter Malle  and the referee  for  their  careful reading  and corrections.

\section{Preliminaries} \label{sec:prelims}

\subsection{Twists through ring automorphisms}
Let $R$ be a commutative ring with identity and let $ \varphi: R \to R$ be a ring automorphism.  For an $R$-module $V$,  the $\varphi$-twist  $V^{\varphi} $ of $V$  is   the $R$-module  which equals $V$ as  group  and where scalar multiplication is given by $ \lambda\cdot v = \varphi^{-1}(\lambda) v $, for $ \lambda \in R$, $ v\in V$.  Denote by $R\textnormal{-Mod}$ the category of  $R$-modules   and by  $\varphi: R\textnormal{-Mod}  \to R\textnormal{-Mod} $ the functor which sends an object $V$ to $V^{\varphi}$ and is the identity on morphisms. Then $\varphi$ is an additive equivalence.

For  $A$ an $R$-algebra,  we denote by   $ A^{\varphi} $   the $R$-algebra   which is equal to  $ A^{\varphi}$ as  an $R$-module  and to $A$ as a ring.  The functor $\varphi: R\textnormal{-Mod}  \to  R\textnormal{-Mod} $  extends to   an  additive equivalence $ \varphi:  A\textnormal{-Mod}  \to A\textnormal{-Mod} $.  However,  the  $R$-algebras  $A$ and $A^{\varphi}$ are not necessarily Morita equivalent. This gives  rise to several  invariants, all of which can be thought of as a measure  of the rationality of $A$ with respect to $\varphi$.

\begin{defi} \label{defi:fnumberproto}  
	Let $R$ be a commutative unital ring and $ \varphi : R \to   R$ an automorphism. Let $A$ be an $R$-algebra.  \vspace{.5ex}
\begin{itemize}  
	\item  The \textit{Morita Frobenius number} \textit{mf}$(A)$  of $A$  with respect to  $\varphi $   is the least positive integer  $m$ (possibly infinity)    such that   $A^{\varphi^m} $ and $A$ are Morita equivalent as $R$-algebras.  \vspace{1ex}

	\item  The \textit{Frobenius number} $f(A)$ of $A$ with respect to $\varphi$  is the least positive integer  $m$ (possibly infinity)  such that $A^{\varphi^m} $ and $A$ are isomorphic as $R$-algebras. \vspace{1ex}

	\item  Let   $  R \subseteq    R'$   be an inclusion of unital  commutative rings and let $ \varphi': R' \to R' $ be  a ring automorphism  extending $\varphi$.  The   \textit{strong Frobenius number} \textit{sf}$(A)$ of $A$ with respect to $\varphi $ and $\varphi'$     is the least    positive integer $m$  (possibly infinity)  such that  there is an $R$-linear isomorphism  $\tau :   A \to   A^{\varphi^m}$  such that the equivalence of module categories induced by $\tau' :  R'\otimes_{R}  A \to  R'\otimes_{R}  A^{\varphi^m} $, the   unique  $R'$-linear extension of $\tau $, sends any simple $R' \otimes_R A$-module $V$  to   $ V^{\varphi'{^m}}\!$.
\end{itemize}
\end{defi} 
It is immediate from the  definitions that for any $\varphi$ and $\varphi'$, \textit{mf}$ (A) \leq f (A) \leq sf (A)   $.  The     use  of   ``Frobenius"  in the terminology goes back to \cite{B/K} where  Morita Frobenius and Frobenius numbers  were defined  for finite dimensional algebras  over fields of positive characteristic with $\varphi$ the standard Frobenius isomorphism (see the following subsections).   The concept of strong  Frobenius numbers   is due to Eaton and Livesey \cite{E/L1}; the related concept of $\varphi$-equivalence  appeared in \cite{C/K}. 

For finite group algebras  and their direct factors, $\varphi$-twists   have another convenient  interpretation. For $\varphi $ and $R$ as above    we denote   by $\varphi : RG \to RG$ the ring automorphism which sends an element $ \sum_{g \in G} \alpha_g g  $ of $RG$ to   $ \sum_{g \in G} \varphi (\alpha_g) g  $. Then  $\varphi^{-1} : RG \to  (RG)^{\varphi} $ is an  isomorphism  of $R$-algebras. Further, for any central idempotent $b$ of   $RG$,   $\varphi (b)$   is a central idempotent   of $RG$  and   $ \varphi^{-1}$  restricts to  an $R$-linear isomorphism    $RG \varphi(b) \cong  (RGb)^{\varphi}$.    

For any function $ \chi : G \to R$ we denote by $^{\varphi} \chi : G \to R$ the function defined by $^{\varphi} \chi(g) = \varphi(\chi(g))$, $ g\in G$.

\subsection{$\ell$-modular systems and Frobenius maps} 
Throughout this paper $\ell$ denotes a prime number. Let $(K, \cO, k)$  be   an $\ell$-modular system with $k \cong \overline \FF_\ell$ such that $\cO$ is absolutely  unramified, i.e., $ J(\cO)=\ell \cO$  (see \cite[Chapters 1,2]{Se} for generalities on complete discrete valuation rings).

Let $ \sigma : k \to k $   be  the   Frobenius automorphism  defined by $\lambda \mapsto \lambda^\ell$,   $\lambda \in k$.  Recall that  by the structure theory of  complete discrete valuation rings  there is a  unique   automorphism  $$ \hat \sigma:   \cO \to \cO $$    lifting the  Frobenius  automorphism $ \sigma $  of $k$.  The  map $\hat \sigma  $  extends uniquely to an automorphism of $K$. Let    $\bar  K $  be  a  fixed  algebraic  closure  of  $K$.  Since all $\ell'$-roots of unity in $\bar K$ belong to $K$, if   $\tau:  \bar  K  \to  \bar K$ is any automorphism of $\bar K$ extending  $\hat \sigma $  then  
$\tau (\zeta )= \zeta^{\ell} $   for every $\ell'$-root of unity  in  $\bar  K$.  The  following fact is well-known, we give a proof for the convenience of the reader.

\begin{lem} \label{lem:extendingfrob}   
	There exists a  field automorphism $\tau : \bar K \to \bar K$  extending  $\hat \sigma $ such that $\tau (\zeta) =\zeta $  for every root of unity $\zeta$ in $\bar K$  of $\ell$-power order. 
\end{lem}

\begin{proof}
	Let  $\zeta \in \bar K$  be  a root of unity  of $\ell$-power order. We  claim that there is a  unique  extension $\tau: K[\zeta]  \to K[\zeta] $  of $\hat \sigma $ such that $\tau (\zeta)= \zeta $.  Indeed, let $\tau:   K[\zeta] \to K[\zeta] $ be any extension of $\hat \sigma $  and suppose that $\tau (\zeta)=  \zeta^i $. Since $\cO$ is unramified, $\textnormal{Gal}(K[\zeta]/K) \cong \textnormal{Aut}(\langle \zeta \rangle ) $ (see for instance \cite [Lemma 3.3]{K/L}).  Thus there exists  $\eta  \in   \textnormal{Gal} (K[\zeta]/K) $ such that $\eta(\zeta) = \zeta^i$. Replacing $\tau$  with  $\eta^{-1} \tau $ results in an  automorphism which  extends $\hat \sigma $ and sends $\zeta $ to itself.  The uniqueness assertion is obvious and the claim follows.  Now let $ K_0 \subset \bar K$ be the union of all subfields of the form   $K[\zeta]$, where $\zeta $ is  a root of unity  of $\ell$-power order.  Then by a standard Zorn's Lemma argument  we have that there is an extension  $\tau : K_0 \to K_0 $ of $\hat\sigma $ such that    $\tau(\zeta) =  \zeta $ for any root of unity of $\ell$-power order  in $K_0 $ and hence in  $\bar K$.  Now any extension of $\tau $ to $\bar K  $ has the desired property.    
\end{proof}

Henceforth,  we  fix an extension    
\[\hat \sigma : \bar K \to \bar K  \]
of $ \hat \sigma $ to $\bar K $  such that $\hat \sigma (\zeta) =\zeta $   for every   root of unity  $\zeta$  in $\bar K$  of $\ell$-power order.

\subsection{Blocks of finite groups and Frobenius numbers} Let  $G$ be a finite group.  By a block of $\cO G$ (respectively a block of $kG$) we mean a primitive idempotent in $Z(\cO G)$ (respectively $Z(kG)$). Recall that the canonical quotient map $\cO\to k$ extends to  a   surjective ring homomorphism   
\[\pi: \cO G \rightarrow k G\]
and $\pi $  induces a bijection  between the set of blocks of $\cO G$ and the set of blocks of $kG$.  Where there is no ambiguity, we will use the term block or $\ell$-block of $G$  to  refer to  either a  block of $\cO G$ or of $ kG$.

We denote by  $\textnormal{Irr}(G)$  the set of  irreducible $\bar  K$-valued characters of $G$  and for  $\chi \in \textnormal{Irr}(G)$, let $e_{\chi}$ denote the central primitive idempotent of $\bar KG$ corresponding to $\chi$.   For $b$ a block of $\cO G$, the set of characters belonging to $b$  or $\pi(b)$ is defined to be 
\[ \mbox{Irr}(b) = \mbox{Irr}(\pi(b))= \{ \chi \in \mbox{ Irr}(G)  \mid  b e_{\chi} = e_{\chi} \}.\]  

\begin{lem}[{\cite[Lemma 2.2]{F}}] \label{lem:charactergalois}
	Let $G$ be a finite group and let $b$  a block of $\cO G$. Then   $\pi(\hat{\sigma}(b)) = {\sigma(\pi(b))}$  and $\textnormal{Irr}(\hat \sigma(b)) = \{ ^{\hat\sigma}{\chi} \mid \chi \in \textnormal{Irr}(b)\}$.
\end{lem}

We  restate the definitions of the previous subsection  by making use of the   identification  of the  twist  of  a block algebra  with   another block    of   the same group algebra. The relevant  ring automorphisms are $\sigma $ and $\hat \sigma $.   Note that when we say that two block algebras  $\cO Gb $ and $\cO Nc$ are Morita equivalent (or isomorphic), we  mean     as $\cO$-algebras and similarly over $k$. 

\begin{defi}\label{defi:MF}
Let $b$ be a block of $\cO G$  and let $\cO Gb $  and $ kG\pi(b)= k\otimes_{\cO} \cO G b  $  be  the  corresponding block algebras.   

\begin{itemize} 
	
	\item The \textit{Morita Frobenius number} of $kG\pi(b)$, \textit{mf}$(kG\pi(b))$, is the minimal positive integer $m$ such that $k G \pi(b)$ is Morita equivalent to $k G \sigma^m (\pi(b))$.\vspace{.5ex}
	
	\item The \textit{Morita Frobenius number} of $\cO Gb $, \textit{mf}$(\cO  Gb )$, is the minimal positive integer $m$ such that $\cO G b$ is Morita equivalent to $\cO G\hat \sigma^m (b)$. \vspace{.5ex} 

	\item  The \textit{Frobenius number} of $kGb$, \textit{f}$( kGb )$,  is the minimal positive integer $m$ such that $k G \pi(b) \cong k G \sigma^m (\pi(b))$.  \vspace{.5ex}

	\item  The \textit{Frobenius number} of $\cO Gb $, \textit{f}$(\cO  Gb)$, is the minimal positive integer $m$ such that $\cO G b \cong \cO G \hat \sigma^m (b)$. \vspace{.5ex}

	\item  The \textit{strong Frobenius number} of $\cO  Gb$, \textit{sf}$(\cO  Gb )$, is the minimal positive integer $m$ such that there exists an $\cO$-algebra isomorphism from $\cO G b$ to $\cO G \hat \sigma^m (b)$ which, when extended to an isomorphism $\bar K G b \rightarrow \bar K G \hat \sigma^m(b)$, induces a bijection on characters given by $\chi \mapsto \,  ^{\hat \sigma{^m}}\!\chi $, for all $\chi \in $ Irr$(b)$.
\end{itemize}  
\end{defi}

The  first and third  definition above are equivalent to the definitions in the previous subsection with $R=k$, $\varphi=\sigma $ and $A=kG\pi(b)$. The second and fourth are equivalent to the definitions in the previous subsection with $R=\cO $, $\varphi=\hat \sigma $ and $A=\cO Gb$.   The  fifth definition  corresponds  to  the definition in the previous subsection with  $ R=\cO $, $R'=\bar K $,  $\varphi=\varphi'=\hat \sigma $, and $A=\cO Gb $. Note that the  isomorphism type of a simple $\bar  KG$-module  is determined  by   its   character  and  that  if $\chi $ is the character of the simple $\bar KG $-module $V$, then $^{\hat \sigma^m}\! \chi  $ is the  character  of  $V^{ \hat \sigma^m }  $ when regarded as  a  $\bar K G$-module   via pull  back  through the  map  $\hat\sigma^{m^{-1}}\!:  (\bar KG )^{\hat \sigma^m} \to  \bar KG  $.   

Since $\cO G$  has only  finitely many   blocks, $ \textit{f }\! (\cO  Gb) $  is   finite for any block $b$.  The first part of the following    proposition shows   that  the same is true for  the other numbers.

\begin{prop}\label{prop:ELresults}  
	Let $G$ and $N$  be finite groups. Let $b$ be a block of $\cO G$ and let $c$  be a block of  $\cO  N$.
	\begin{enumerate}[(i)]
		\item  $  \textit{mf }\! (kG\pi(b) ) \leq  \textit{mf }\! (\cO Gb) \leq  \textit{f }\!(\cO  Gb) \leq  \textit{sf }\!(\cO G b) \leq |D|^2! \, \textit{f }\!(\cO Gb)$. 
		\item  If $\cO Gb $ and $\cO  Nc $ are  Morita equivalent, then $\textit{mf }\!(\cO Gb) = \textit{mf }\!(\cO  Nc)  $. If  $kG\pi(b )$ and $kN\pi(c)  $ are  Morita equivalent, then $\textit{mf }\!( kG\pi(b) ) = \textit{mf }\!(kN \pi(c))  $.
		\item  If $\cO Gb $ and $\cO  Nc $ are  Morita equivalent,   then $\textit{sf }\!(\cO Gb) =  \textit{sf }\!(\cO Nc)$.
	\end{enumerate}
\end{prop}

\begin{proof}
The first two inequalities in  part (i)  follow directly from the definitions and the last two inequalities are \cite[Proposition 2.3 (i)]{E/L}. If $M$  is  an $\cO  Gb \otimes _{\cO} (\cO Nc)^{op} $-module,  then $M^{\sigma^m}$ is  an $(\cO  Gb)^{\hat \sigma^m}  \otimes_{\cO}  ((\cO Nc)^{op})^{\hat \sigma^m}  \cong (\cO  Gb \otimes_{\cO}  (\cO Nc)^{op} )^{\hat \sigma^m}  $-module and  $M$ induces a Morita  equivalence between  $\cO Gb$  and $\cO N c $ if and only if  $M^{\hat \sigma^m}$  induces a Morita equivalence between $(\cO Gb)^{\hat \sigma^m} $  and $ (\cO N c)^{\hat \sigma^m}  $.  This proves part (ii) over $\cO$. The proof  over $k$ is identical. Part (iii) is \cite[Proposition 2.3 (ii)]{E/L}. 
\end{proof}

We will make repeated use of the following result of Linckelmann  \cite[Theorem 1 and Proposition 2]{Li}.

\begin{prop} \label{prop:cyclicklein} Let $ G$ be a finite group and let $ b$ be a block of $\cO G$. If the  defect groups of $\cO Gb $ are cyclic, or if $\ell=2 $ and the defect groups of $\cO Gb$ are  Klein four groups, then $\textit{sf }\!(\cO Gb ) =1 $.  
\end{prop}

\subsection{Morita Frobenius numbers and basic algebras of blocks of $kG$.}

\begin{lem}[{\cite[Lemma 2.1]{K}}]\label{lem:definedoverFl}  
	Let $A$  be  a finite dimensional  $k$-algebra. Then $A$ is defined over $\FF_{\ell^m}$ if and only if $A \cong A^{\sigma^m}$.
\end{lem}

Two finite dimensional   $k$-algebras are Morita equivalent if and only if their basic algebras are isomorphic (see  \cite[Section 4.9]{L2}  for generalities on basic algebras).   If   $A_0$ is a basic algebra of a finite dimensional $k$-algebra $A$, then $ A_0^{\sigma^m} $ is a basic algebra   of $A^{\sigma^m}$.   We thus obtain the following.

\begin{lem}\label{lem:definedoverFl2}  
	Let  $A$  be a finite  dimensional  $k$-algebra. The Morita Frobenius number of $A$ is the least positive integer $m$ such that the basic algebras of $A$ are defined over $\FF_{\ell^m}$.
\end{lem}

\subsection{Covering and dominating blocks.}   
For $\chi \in \textnormal{Irr}(G)$ we let $b(\chi)$ denote the block of $\cO G$ containing $\chi$.   If $N \lhd G$ then for any $\theta \in \textnormal{Irr}(N)$ and any $\chi \in \textnormal{Irr}(G)$ we use the following notation for the set of irreducible characters of $G$ covering $\theta$, and the set of irreducible characters of $N$ covered by $\chi$, respectively. 
\[ \textnormal{Irr}(G \,|\, \theta) = \{ \psi \in \textnormal{Irr}(G) :\theta \textnormal{ is an irreducible constituent of } \psi_{N} \} \]
\[ \textnormal{Irr}(N \, |\, \chi) = \{ \psi \in \textnormal{Irr}(N) :\psi \textnormal{ is an irreducible constituent of } \chi_{N} \} \]

\begin{lem}
	\label{lem:coveringsameblock}
	Suppose $N \lhd G$ are finite groups such that $G/N$ is abelian and let $
	\theta \in \textnormal{Irr}(N)$. Then for any pair of characters $\chi_1, \chi_2 \in \textnormal{Irr}(G \,|\, \theta)$, there exists a linear character $\eta \in $ \textnormal{Irr}$(G/N)$ such that $\chi_2 =  \chi_1\eta.$ Moreover, $\cO G b(\chi_1) \cong \cO G b(\chi_2)$ as $\cO$-algebras.
\end{lem}

\begin{proof}
	The first part follows from \cite[Lemma 13.21]{D/M}. To show the second part, let $\eta \in \Irr(G/N)$ such that $\chi_2 = \chi_1 \eta$, and let $\eta'$ denote the $\ell'$-part  of $\eta $. Then $\eta' $ takes values in $\cO^{\times} $  and  the  map given by 
	\[\sum_{g \in G} \alpha_g g  \mapsto  \sum_{g \in G} \alpha_g g \eta\left(g^{-1}\right), \  \  \sum_{g \in G} \alpha_g g  \in \cO G  \] 
	is an $\cO$-algebra automorphism of $\cO G$ which restricts  to  an isomorphism between $\cO G b(\chi_1)$  and  $\cO G b(\chi_1 \eta')$.  Since $ \eta  $ and $\eta'$ agree on $\ell'$ elements of $G$, 
	\[ b (\chi_2) = b (\chi_1\eta) = b (\chi_1 \eta' ) \]  proving the second assertion.
\end{proof}

Recall that  if $ Z $ is  a normal subgroup  of a finite group  $G$,  then a   block   $\bar b$ of $\cO (G/Z) $  is said to be dominated by a  block $b$ of $ \cO G$ if  $  \bar b \mu (b) \ne 0 $, where  $\mu $ is the  $\cO$-algebra  surjection $ \cO G \to \cO (G/Z) $ induced by     the canonical surjection  $ G \to G/ Z$.    

\begin{lem}\label{lem:domblocks}
	Let $Z$  be a   normal   subgroup   of $G$    and  let  $\mu : \cO G \to \cO (G/Z)$  be the  $\cO$-algebra homomorphism  induced by the canonical surjection map from $G$ to $G/Z$. Then for any block $b$ of $\cO G$, either  $\mu(b) = 0 $ or $\mu (b)  $ is a sum of blocks of $\cO (G/Z)$. Moreover, we have the following.
	\begin{enumerate}[(i)] 
		\item For any block  $c$ of $\cO (G/Z) $ there is a unique block $b$ of $\cO G$   such that   $ c = \mu ( b )$.
		\item If $Z$ is an $\ell' $-group and $ \mu (b)  \ne 0 $, then $\mu$   restricts to an isomorphism $\cO G b \cong   \cO(G/Z) \mu (b)  $. 
		\item If $ Z$  is an $\ell$-group and $ Z \leq Z(G)$,  then $\mu$    induces  a bijection between the set of blocks of $\cO G $ and the set of blocks of $\cO (G/Z) $.  
		\item   The map $\mu $ commutes with the action of  $\hat \sigma $.
	\end{enumerate}
\end{lem}

\begin{proof}  The first claim and part (i)  follow from the fact that $\mu $ is a  surjective homomorphism of 
$ \cO$-algebras  whence the image of any  central idempotent is either zero or a central idempotent.   Suppose that  $Z$ is an $\ell' $-group  and  that $ \mu (b) \ne 0$. Then $\Irr(b)  = \Irr(\mu(b))$  where we regard $\Irr (G/Z) $  as  a subset of $\Irr(G)$ via inflation \cite[Chapter 5, Theorem 8.8]{N/T}.   Thus,  $\cO G b $ and     $\cO(G/Z) \mu (b) $ have the same $\cO$-rank and  the restriction of $\mu $ to $ \cO G b $ is injective.   Since this restriction is also surjective, we obtain (ii).     For a proof of (iii), see \cite[Chapter 5, Theorem 8.11]{N/T}.     Part (iv) is immediate from the definitions.
\end{proof}

\begin{lem}\label{lem:exttoC2xC2}
	Suppose $N \lhd G$ are finite groups such that $G/N \cong C_2 \times C_2$. Then for any $G$-stable linear character $\tau \in $ \textnormal{Irr }$(N)$, its square $\tau^2$ extends to a linear character of $G$. 
\end{lem}

\begin{proof}
	Let $H \leq G$ be such that $G/H \cong C_2$ and $H/N \cong C_2$. Since $H/N$ is cyclic and $\tau$ is $G$-stable, and therefore $H$-stable, $\tau$ extends to a linear character $\hat \tau$ of $H$. 
	
	Suppose that $g \in G$. Then $^g \tau = \tau$ so both $\hat \tau$ and $^g \hat \tau$ are elements of $\textnormal{Irr}(H \,| \,\tau) = \textnormal{Irr}(H \,|\, ^g \tau)$. Thus by Lemma~\ref{lem:coveringsameblock}, there exists a linear character $\la \in $ Irr $( H/N) $ such that $^g \hat \tau = \la \hat \tau $. Since $H/N \cong C_2$, $\la^2 = 1$. 
	
	As $\tau$ is a linear character, $\hat \tau^2$ is an extension of $\tau^2$ to $H$. Therefore $^g (\hat \tau ^2) = (^g \hat \tau )^2 = (\la \hat \tau)^2 = \hat \tau ^2$. It follows that $\hat \tau ^2 $ is $G$-stable, and therefore $\hat \tau^2$ extends to $G$. Hence, $\tau^2$ extends to $G$. 	
\end{proof}

The following easy lemma will be used in the next section. 

\begin{lem} \label{lem:lingalois}     
	Let $G$ be a finite group and let $ H$ be a subgroup of $G$.    If $\theta : H \to  \bar K^{\times} $ is a linear character  of $\ell'$-order, then  $^{\hat\sigma}\theta = \theta^{\ell} $.
\end{lem}

\begin{proof}   
	Since $\theta  $ and  $ ^{\hat \sigma} \theta $ have the same kernel,  we may assume that  $H$ is an abelian $\ell'$-group.  Then  for any    $x \in H$, we have 
	\[ ^{\hat \sigma} \theta (x) = \sigma (\theta(x) )  =  \theta(x)^{\ell} =\theta^{\ell} (x) . \qedhere \]
\end{proof}  

The following descent result is a consequence of Rickard's theorem   on lifting splendid Rickard equivalences from characteristic $\ell$ to characteristic $0$.

\begin{lem} \label{lem:BDRdescent} 
	Let $ (K', \cO', k') $ be an $\ell$-modular system and let $\cO_0 :=  W(k') \leq \cO' $ be the Witt vectors of $k'$ in $ \cO' $.   Let $b$ a central idempotent of $\cO_0 G $, $H$ a finite group and $c$ a central idempotent of $\cO_0 H$. Suppose that  $C$ is  a  splendid Rickard complex of $(\cO'Gb, \cO'Hc )$-bimodules.  Then there exists a splendid Rickard complex $C_0$ of $(\cO_0Gb, \cO_0Hc)$-bimodules such that the following holds.
	\begin{enumerate}[\rm (i)] 
		\item $ C \cong  \cO'\otimes _{\cO_0}  C_0   $  as a complex  of $(\cO'Gb, \cO'Hc )$-bimodules.
		\item For all $i$, $H^i(C) \cong \cO'\otimes_{\cO_0}   H^i\! (C_0)\!$ as  $ (\cO'Gb, \cO'Hc )$-bimodules. 
	\end{enumerate}
	If in addition, $H^d(C)$ induces a Morita equivalence between $\cO'Hc$ and $\cO'G b$ for some $d$, then $H^i(C_0)$ induces a Morita equivalence between $\cO_0 Hc\!$ and $\cO_0 Gb\!$ for \break all $i$. 
\end{lem} 

\begin{proof}  
	Since $\cO_0 $ and $ \cO$ have the same residue field the existence of $C_0$ and statement (i) are immediate from the    existence and uniqueness assertions of  \cite[Theorem~5.2]{Ri}. More specifically, the complex $ \bar C:=  k' \otimes_{\cO'} C_0  $ is a splendid Rickard complex of  $ ( kG \bar b, k H  \bar c  )$-bimodules  where $ \bar b $ and $\bar c$ denote the images    of  $b$ and $c$ in $kG$ and $kH$ respectively. By \cite[Theorem~5.2]{Ri}, there exists a splendid Rickard complex  $C_0$ of $(\cO_0Gb, \cO_0Hc)$-bimodules such that $k' \otimes_{\cO_0} C_0 \cong \bar C$. Note that Rickard's theorem is stated under the assumption that the fraction field of $\cO_0 $ be a splitting field for $G$ and $H$ but the proof does not use this -- the only ingredient  is the completeness of $\cO_0$  and the fact that $\ell$-permutation modules lift   (uniquely) from $k$ to $\cO_0 $. By extension of scalars, $ \cO' \otimes_{\cO_0} C_0$ is a splendid Rickard complex of $(\cO'Gb, \cO'Hc)$-bimodules. Further,
	\[ k \otimes _{\cO'} ( \cO' \otimes_{\cO_0}   C_0  ) \cong   k\otimes_{\cO_0} C_0  \]	
	as $(\cO'Gb, \cO'Hc )$-bimodules. Statement (i) now follows as by Theorem \cite[Theorem~5.2]{Ri}, $C$ is the unique splendid Rickard complex lifting $\bar C$.
 
	Statement (ii) follows from (i) since $\cO_0 \subseteq \cO'$ is a flat extension. The final statement  follows from  an application of  the Noether-Deuring theorem  (see for instance \cite[Prop.~4.5]{K/L}).
\end{proof} 

\section{Galois actions and  Lusztig series }\label{sec:alggroups}
We continue with the  notation of the previous section. There is a  canonical  embedding  (as valuation fields) of ${\mathbb Q}_{\ell}$ in $K$. Let $\bar {\mathbb Q}_{\ell}$ denote the algebraic closure of ${\mathbb Q}_{\ell}$ in $\bar K$. We fix a prime $p$ different from $\ell$, a group isomorphism  $\iota: ({\mathbb Q}/{\mathbb Z})_{p'} \to \overline{\FF}_p^{\times}$ and an injective group homomorphism $\jmath: ({\mathbb Q}/{\mathbb Z}) \to   \bar {\mathbb Q}_{\ell}^{\times}$.

\subsection{Characters of tori and duality}\label{subsec:alggroups}
Let $\bG$ be a connected reductive algebraic group defined over $\overline{\FF}_p$ with   Frobenius   endomorphism $F: \bG \rightarrow \bG$ (see  Remark~\ref{rem:reesuzuki}).   We fix an $F$-stable maximal torus $\bT_0 $ of $\bG $,  a  triple  $(\bG^*, \bT^*_0, F )$ dual to the triple $(\bG, \bT_0, F) $ and  an $F$-equivariant isomorphism $ X (\bT_0) \cong Y (\bT^*_0) $ which sends simple roots to simple coroots as in \cite[Definition 13.10]{D/M}, where $ X (\bT_0)$ denotes the set of characters of $\bT_0$ and $ Y (\bT^*_0) $ denotes the set of cocharacters of $\bT^*_0$. We let $\nabla(\bG, F)$ denote the set of pairs $(\bT, \theta) $ where $\bT$ is an $F$-stable maximal torus of $\bG$ and $\theta: \bT^F \to \bar{\mathbb Q}_{\ell}$ is a linear character. Dually, let  $\nabla^*(\bG^*, F)$ denote the set of pairs $(\bT^*, s) $ where $\bT^*$ is an $F$-stable maximal torus of $\bG^* $ and $s \in \bT^{*F}$.  

The choice  of $\iota$, $\jmath$ and the isomorphism $X(\bT_0) \cong Y(\bT_0^*)$ above determine a bijection \cite[11.15, Proposition 13.13]{D/M}
\begin{equation}\label{eq:nabladuality}  
\nabla(\bG, F)/\bG^F \rightarrow \nabla^*(\bG^*, F)/\bG^{*F}. \end{equation}  
If $(\bT, \theta)\in \nabla(\bG, F) $ and $(\bT^*, s) \in \nabla^*(\bG^*, F) $  then we write $(\bT, \theta) \stackrel{\bG} {\leftrightarrow} (\bT^*, s)$ if the classes of $(\bT, \theta)$  and $(\bT^*, s)$ correspond under  the above bijection. We write $\bT \stackrel{\bG} {\leftrightarrow} \bT^*$ if $\bT$ is an $F$-stable maximal torus of $\bG $ and $\bT^*$ is an $F$-stable maximal torus  of $\bG^*$ such that $(\bT, \theta)\stackrel{\bG} {\leftrightarrow}   (\bT^*, s)$ for some $\theta $ and some $s$. For $s$  a semisimple element of $\bG^{*F}$ we denote by $\nabla(\bG, F, s)$ the  set of all pairs $(\bT, \theta)$ such that $ (\bT, \theta) \stackrel{\bG} {\leftrightarrow}  (\bT^*, s)$ for some  $F$-stable maximal torus $\bT^* $ of $\bG^*$ containing $s$.  
 
The following lemma lists some well-known properties of bijection (\ref{eq:nabladuality}). We give a short indication of the proof for the convenience of the reader. Let $$ \tau: \bG _{sc} \to [\bG, \bG] $$ be  a simply-connected covering (see \cite[Section 8.1]{C/E3}).

\begin{lem} \label{lem:duality}  
	Let $s$ be a semisimple element of $\bG^{*F} $. 
	\begin{enumerate} [(i)]  
		 \item If $(\bT, \theta)\in \nabla(\bG, F,s)$, then for any $n \in {\mathbb N} $, $(\bT, \theta^n)\in \nabla(\bG, F,s^n ) $.	
		 \item There exists an isomorphism $Z(\bG^{*})^F \to  \Irr  (\bG^F/\tau(\bG_{sc}^F)), (t \mapsto \hat t)$ such that for all $t \in Z(\bG^{*})^F$, $ \nabla(\bG, F, t) = \{ (\bT, \hat t_{\bT^F}) \}$ where $\bT$ runs over all $F$-stable maximal tori of $\bG$.
	\end{enumerate}
\end{lem}   

\begin{proof}    
	Let  $\bT $ be an $F$-stable maximal torus of $\bG$. By   construction  of the bijection (\ref{eq:nabladuality}), if  $\bT^*$ is an $F$-stable maximal torus of $\bG^*$ such that $\bT  \stackrel {\bG}{ \leftrightarrow} \bT^* $, then there exists an  isomorphism  $\zeta:  \bT^{*F}  \to  \mbox{Irr}( \bT^F)$, such that $(\bT, \zeta(t))  \stackrel {\bG} {\leftrightarrow}  (\bT^*,  t )$ for all $t\in \bT^{*F} $ (see \cite[(8.15)]{C/E3}). Now suppose that $(\bT, \theta) \stackrel{\bG}{\leftrightarrow} (\bT^*,  s)$. Then $(\bT,  \theta) $ is $\bG^F$-conjugate to $(\bT,  \zeta(s)) $, hence $(\bT,  \theta^n) $ is $\bG^F$-conjugate to $(\bT, \zeta(s^n)) $ and $(\bT,  \zeta(s^n)) \stackrel{\bG}{\leftrightarrow}(\bT^*, s^n)$ proving (i).
   
	By \cite[(8.19)]{C/E3}, there exists an isomorphism   $Z(\bG^{*})^F  \to  \Irr  (\bG^F/\tau(\bG_{sc}^F)), ( t \mapsto  \hat t)$ such that for any $\bT $, $\bT^*$ and $\zeta $ as above, and any $t \in Z(\bG^{*})^F \leq \bT^{*F}$, we have that $\hat  t_{\bT^F}= \zeta (t)$ where we view $\hat t $ as a character of $\bG^F$ via pull back through $\bG^F \rightarrow \bG^F/\tau(\bG_{sc}^F)$ (see \cite[Proposition 5.11(ii)]{D/Lpaper} for the independence of the isomorphism). It follows that for any $t \in  Z(\bG^{*})^F$, $(\bT, \hat t_{\bT^F})=(\bT, \zeta(t)) \stackrel{\bG}{\leftrightarrow} (\bT^*,  t )$, and hence $(\bT, \hat t_{\bT^F} ) \in \nabla(\bG, F, t)$. 

	Now suppose that $(\bT, \theta) \in \nabla ( \bG, F, t) $, say $ (\bT, \theta)  \stackrel{\bG}{\leftrightarrow} (\bT^*,  t)  $.  Then   $(\bT,  \theta) $ is $\bG^F$-conjugate to $(\bT,  \zeta(t)) =  (\bT, \hat t_{\bT^F} )  $. But $\hat t$ is a linear character of $\bG^F$, so it is $\bG^F$-stable, hence   $(\bT,  \theta)  =  (\bT, \hat t_{\bT^F} )  $,  proving part (ii). \qedhere

\end{proof}

\begin{defi} \label{defi:rH}
Let $H$ be a finite group and let $h \in H$  be an $\ell'$-element. We denote by $a_{H}(h)$ the order of $h$ modulo $Z(H)$, and by $r_{H}(h)$ the multiplicative order of $\ell$ modulo $a_{H}(h)$. 
\end{defi}

\pagebreak
\begin{lem}\label{lem:thetatildeinG} 
	Let $s \in \bG^{*F}$ be a semisimple $\ell'$-element and let $r = r_{\bG^{*F}}(s)$. 
	\begin{enumerate}[(i)]  
		\item Let  $ (\bT,   \theta ) \in \nabla(\bG, F) $. Then $ (\bT,   \theta ) \in \nabla(\bG, F,s) $  if and only if $ (\bT, ^{\hat \sigma}\theta ) \in \nabla(\bG, F,s^\ell) $.
		\item  Let $ \tau$ denote the linear character of $ \bG^F$ corresponding to $s^{\ell^r-1} $  by the isomorphism in Lemma~\ref{lem:duality} (ii).  Then  $^{\hat \sigma ^r} \theta  = \theta  {\tau}_{ \bT^F}  $  for all $( \bT,  \theta) \in \nabla (\bG, F, s)$.  
	\end{enumerate}
\end{lem}

\begin{proof}     
	Part (i)  follows from Lemma~\ref{lem:lingalois} and   Lemma~\ref{lem:duality} (i). To prove (ii), note that by definition of $r$,  $s^{\ell^r-1} \in  Z(\bG^{*F})= Z(\bG^*)^F$, hence  $\tau$ is well-defined. Let  $ (\bT, \theta ) \in \nabla(\bG, F, s) $.  By Lemma~\ref{lem:duality} (i), $(\bT, \theta^{\ell^r-1} ) \in (\bG, F, s^{\ell^r-1})$.  Hence, by    Lemma~\ref{lem:duality} (ii), $\theta^{\ell^r-1} =\theta  {\tau}_{ \bT^F}$. Thus 
	\[^{\hat \sigma ^r} \theta = \theta^{\ell^r}   = \theta \theta^{\ell^r-1} = \theta  {\tau}_{ \bT^F}\]  
 where the first equality holds by Lemma~\ref{lem:lingalois}.
 \end{proof}

\subsection{ Regular embeddings.}   
We fix a regular embedding $i: \bG \rightarrow \widetilde \bG$  (see \cite[Section 15.1]{C/E3}) and let $F$ denote a Steinberg morphism on $\widetilde \bG$ compatible with the Steinberg morphism $F$ on $\bG$. The map $\bT \mapsto  Z(\widetilde \bG) \bT$ induces a bijection  between the set of ($F$-stable) maximal tori of $\bG$  and the  set  of  ($F$-stable)   maximal tori of $\widetilde \bG $. The inverse bijection  is given by  $\widetilde  \bT  \mapsto    \widetilde \bT \cap \bG$. Set $\widetilde \bT_0  =  Z(\widetilde \bG) \bT_0  $. Let $(\widetilde \bG^*,  \widetilde \bT_0^*,  F)$ be dual to $(\widetilde\bG, \widetilde\bT_0, F)$ and $i^*: \widetilde \bG^* \rightarrow \bG^*$   be a surjection dual to $i$ with $ i^*(\widetilde  \bT_0^*)= \bT_0^* $. Fix an isomorphism $ X(\widetilde \bT_0) \to Y(\widetilde \bT_0^*)$ lifting the isomorphism $ X(\bT_0) \to Y(\bT_0^*)$.

\begin{lem} \label{lem:torusembed} 
	Let $\tilde s $ be a semisimple element of $\widetilde \bG^{*F} $ and let $s = i^*(\tilde s)$. Let $\widetilde \bT$ be an $F$-stable maximal torus of $\widetilde \bG$, $\tilde \theta: \mathrm {Irr} (\widetilde \bT^F) \to \bar{\mathbb{Q}}_{\ell} $ a linear character of $\widetilde \bT^F $ and let $\theta$ denote the restriction of $\tilde \theta$ to $\bT^F$. If $(\widetilde \bT, \tilde \theta)  \in \nabla(\widetilde\bG, F,\tilde s)  $, then  $(\bT, \theta)\in \nabla(\bG, F,s)$.
\end{lem} 

\begin{proof}  
	See \cite[Lemme~9.3 (a)]{B}.
\end{proof}

We also  record for future use the following fact.
\begin{lem}\label{lem:rgood}  
	Let $\tilde s \in  \widetilde  \bG^{*F}$ be a semisimple element  and set $s=i^*(\tilde s)$. Then $a_{\widetilde \bG^{*F}}  (\tilde  s) =  a_{\bG^{*F}  }(s)  $  and $r_{\widetilde \bG^{*F}}(\tilde s) =  r_{\bG^{*F}}(s)$. 
\end{lem}

\begin{proof} 
	Let $a= a_{\bG^{*F}}(s)$. Then $i^*(\tilde s ^a)  = s^a \in   Z(\bG^{*F}) \leq Z(\bG^*)$. Now since $i^*: \widetilde \bG^* \to \bG^*$ is a surjective homomorphism with kernel a central torus of $\widetilde \bG^*$ (see \cite[Proposition~2.5]{B}), $i^{*-1}( Z(\bG^*)) = Z(\widetilde \bG^*)$. Hence $\tilde s^ a \in Z(\widetilde \bG^*) \cap \widetilde \bG^{*F} =Z(\widetilde \bG^{*F})$ showing that $a_{\widetilde \bG^{*F}}(\tilde s) \leq a_{\bG^{*F}  }(s)$. The reverse inequality is immediate from the fact that the restriction of $i^*$ to $\widetilde \bG^{*F}$ is surjective (see \cite[Corollaire~2.7]{B}).
 \end{proof}

\subsection{Lusztig series}
For  a  semisimple element $s$ of $ \bG^{*F}\!$, we  denote by   $\cE(\bG^F, s) \subseteq  \mbox{Irr} (\bG^F) $ the  rational  Lusztig series corresponding to the $\bG^{*F}$-conjugacy class of $s$. Recall that  $\cE(\bG^F, s)$ is the set consisting of the irreducible  characters for which $\langle \chi, R_{\bT}^{\bG}(\theta) \rangle \ne  0 $ for some $(\bT, \theta) \in \nabla (\bG, F, s) $. By results of Lusztig, the series $\cE(\bG^F, s) $ as $s$ runs over a set of representatives of $\bG^F $-conjugacy classes of semi-simple elements of $\bG^F$  partition $\Irr(\bG^F)$.

\begin{lem}\label{lem:lusztiggalois}  
	Let  $ \chi \in \Irr(\bG^F)$, $ (\bT, \theta) \in \nabla (\bG, F) $ and let  $s \in \bG^{*F}$ be a semisimple $\ell' $-element.
	\begin{enumerate} [(i)] 
		\item  For any $\varphi \in  \textnormal{Aut}(\bar K)$, $ \langle ^{\varphi} \chi  \, ,  \,  \!R_{\bT}^{\bG} (\,^{\varphi} \theta) \rangle =  \langle  \chi \, ,  \,  \!R_{\bT}^{\bG} (\theta) \rangle  $.
		\item $ \chi \in \cE (\bG^F, s) $ if and only if $ \,^{\hat \sigma} \chi \in \cE(\bG^F, s^{\ell} ) $. 
	\end{enumerate}
\end{lem} 

\begin{proof}   
	By the character formula \cite[Proposition 12.2]{D/M}, $\,^{\varphi} R_{\bT} ^{\bG} (\theta) =  R_{\bT} ^{\bG} (^{\varphi} \theta) $. Hence,
	 \[ \langle \chi \,, \, R_{\bT}^{\bG}( \theta) \rangle = \,^{\varphi} \langle \chi \, , R_{\bT}^{\bG}( \theta) \rangle=
	 \langle {^\varphi} \chi  \,,  \,^{\varphi} R_{\bT}^{\bG}(\!  \theta ) \rangle = \langle ^{\varphi} \chi \,   , \, R_{\bT}^{\bG}( ^{\varphi}\theta) \rangle 
	\]
	where the second equality holds since $R_{\bT}^{\bG} (\theta) $ is a generalised character. This proves (i). Part (ii) follows from (i) and Lemma~\ref{lem:thetatildeinG} (i).
\end{proof}

The next  lemma recalls   a consequence of the Jordan decomposition of characters. 
Recall that for a ${\bar {\mathbb Q}_{\ell}}$-valued class function  $\chi $ of $\bG^F$ the uniform projection of $\chi$ is the orthogonal projection (with respect to $\langle \, , \, \rangle$) of $\chi $ onto the subspace of class functions generated by  $R_{\bT} ^{\bG} (\theta)$, $ (\bT, \theta) \in \nabla (\bG, F)$.       

\begin{lem}\label{lem:uniformprojections} Suppose that $ Z(\bG)$ is connected   and let $s$ be  a semi-simple element of ${\bG^*}^F\!$.    Suppose that     $ C_{\bG^*}(s) $    has all   classical  components. Then each $\chi \in \cE(\bG^F, s) $ is uniquely determined by its uniform projection.   
\end{lem}  

\begin{proof}  The hypothesis   implies that   the unipotent characters of $C_{\bG^*}(s)^F$ are uniquely determined by their uniform projections  \cite[8.1(a)]{Lu88}.  Then   by Jordan decomposition   (\cite[4.23] {Lu}, see also  \cite[Theorem 15.8]{C/E3})  each $\chi \in \cE  (\bG^F, s )$ is uniquely determined by its uniform projection.
\end{proof} 

\begin{lem}\label{lem:conjofchi}   
 Let    	 $s \in \bG^{*F}$  be   a semisimple  element and $\varphi \in \mbox{Aut}(\bar K)$. Suppose that $C_{\bG^*}^{\circ} (s)$ has only classical components.	Suppose further that there exists $\tilde s \in \widetilde \bG^{*F}$ with $i^* (\tilde s) = s$ and  a linear character $\tilde \eta \in \textnormal{Irr}(\widetilde \bG^F)$ such that $^\varphi {\widetilde \theta } = \widetilde \theta \tilde \eta_{\widetilde \bT^F}$ for all $(\widetilde \bT, \widetilde \theta) \in \nabla (\widetilde \bG, F, \tilde s)$.  Then for any $\chi \in \cE(\bG^F, s)$, there exists an $x \in \widetilde \bG^F$ such that 
	\[ ^\varphi \chi = {^x{( \chi \tilde \eta_{\bG^F})}}.\]
\end{lem}

\begin{proof}
	Let $\chi \in \cE(\bG^F, s)$ and let $\tilde \chi \in \textnormal{Irr}(\widetilde \bG^F  \,  | \,  \chi) \cap   \cE(\widetilde \bG^F, \tilde s)$ (for the existence of $\tilde \chi $ see \cite[Proposition 11.7]{B}). Let $(\widetilde \bT, \tilde{\theta}) \in \nabla (\widetilde \bG, F, \tilde {s})$. By assumption,
	\[ \langle \tilde \chi \, , \, R_{\widetilde \bT}^{\widetilde \bG}(\tilde \theta) \rangle 
	= \langle ^{^\varphi}\!\tilde \chi\, , \, \!R_{\widetilde \bT}^{\widetilde \bG}(^{^\varphi} \! \tilde \theta ) \rangle 
	= \langle ^{^\varphi}\!\tilde \chi\, , \, \!R_{\widetilde \bT}^{\widetilde \bG}(\tilde \theta \tilde \eta_{\widetilde \bT^F}) \rangle, \]
	where the first equality holds by Lemma~\ref{lem:lusztiggalois}. On the other hand, since $\tilde \eta \in \textnormal{Irr} (\widetilde \bG^F)$ is linear,  
	\[ \langle \tilde \chi\, ,\,  R_{\widetilde\bT}^{\widetilde\bG}(\tilde \theta) \rangle 
	= \langle \tilde\chi \tilde \eta\, ,\,  R_{\widetilde \bT}^{\widetilde \bG}(\tilde \theta) \tilde \eta \rangle 
	= \langle \tilde \chi \tilde \eta\, , \, R_{\widetilde \bT}^{\widetilde \bG}(\tilde \theta \tilde\eta_{\widetilde\bT^F}) \rangle,\]
	where now the second equality holds by the character formula \cite[Proposition 12.2]{D/M}. Since $ \langle \tilde \chi\,  , \, R_{\widetilde \bT}^{\widetilde \bG}(\tilde \theta) \rangle = 0$ for all $(\widetilde \bT, \tilde \theta) \in \nabla (\widetilde \bG, F) \setminus \nabla (\widetilde \bG, F, \tilde s)$, it follows that $^{\varphi} \tilde \chi$ and $\tilde \chi \tilde \eta$ have the same uniform projections. By definition  $Z(\widetilde \bG^*)$ is connected, so $C_{\widetilde \bG^*}(\tilde s)$ is connected. Since  $C_{\bG^*}^{\circ}(s)$  has all classical components the  same is true of  $C_{\widetilde \bG^*}(\tilde s)$. Hence by Lemma~\ref{lem:uniformprojections}, $\,^{\varphi}\tilde \chi =  \tilde \chi \tilde  \eta $. It follows that $^{\varphi}\chi$ and $\chi \tilde \eta_{\bG^F}$ are both elements of $ \textnormal{Irr}(\bG^F \, | \, ^{\varphi} \tilde \chi) = \textnormal{Irr}(\bG^F \,|\, \tilde \chi \tilde  \eta )$, proving the result. 
\end{proof}

\subsection{Blocks.} 
For $s \in \bG^{*F}$ a semisimple $\ell'$-element  we denote by $\cE_{\ell} (\bG^F, s) $ the  union of the rational Lusztig series $\cE(\bG^F,  t) $, where $ t $ runs over a set of representatives of $\bG^{*F}$-conjugacy classes of the semisimple elements of $\bG^{*F}$ whose $\ell'$-part is $\bG^{*F}$-conjugate to $s$. We recall that  by results of Brou\'{e}-Michel and Hiss \cite[Theorem 9.12]{C/E3}, $\cE_{\ell} (\bG^F, s)$ is the union of the irreducible characters in some subset of the set of blocks of $\cO \bG^F$ and if $b$ is an $\ell$-block of $\cO \bG^F $ such that $\mbox{Irr}(b) \subseteq \cE_{\ell}(\bG^F,  s)$, then  $\mbox{Irr}(b)\cap \cE(\bG^F, s) \ne \emptyset $. If $\mbox{Irr}(b) \subseteq \cE_{\ell}(\bG^F, s)$  we  say ``$b$ is in $ \cE_{\ell}(\bG^F, s)$" and write $b \in \cE_{\ell}(\bG^F,s)$.

Since the largest order of semisimple elements of the groups $\bG^{*F^n}$ become arbitrarily large as $ n \to \infty $,  it is easy to see from the above discussion that there is no finite subfield  $k_0$ of $k$ such that $ b \in k_0 \bG^{F^n}$ for  all $n$ and all blocks $b$ of $kG^{F^n} $. We give an explicit example below.

\begin{exa} \label{ex:idempotentcoefficients} 
	Let $\bG = \mathrm{PGL}_r (\overline \FF_p)$, $r\geq 2$,  $F:\bG \to \bG$ be the standard Frobenius over ${\mathbb F}_p$  with $\bG^{F^n} \cong {\rm PGL}_r(p^n) $, $\bG^* = \mathrm {SL}_r (\overline \FF_p)$, $\bG^{*F^n} \cong \mathrm {SL}_r (p^n)$ for $n \in {\mathbb N}$. Choose a strictly increasing sequence of positive integers  $(n_i)_{i \in {\mathbb N}} $ such that $|\bG^{*F^{n_i}}|_{\ell} = |\bG^{*F^{n_j}}|_{\ell}$ for all $i, j$ and set $G_i = \bG^{F^{n_i}}$, $G_i^* = \bG^{*F^{n_i}}$. For each $i$, let $\lambda_i \in \FF_{p^{n_i}}$ be an element of order $(p^{n_i}-1)_{\ell'}$ and let $ s_i \in G_i^*$ be   a diagonal matrix with diagonal entries $\lambda_i, \lambda_i ^{-1} , 1, \ldots, 1$. Let $d_i$ be the smallest positive integer such that $s_i^{\ell^{d_i} }$ is $\bG^*$-conjugate to $s_i$. Since $(p^{n_i}-1)_{\ell'}\to \infty$ as $i \to \infty $, we have that $ d_i \to \infty $ as $i\to \infty $.       

	For each $ i\in {\mathbb N} $, choose a  block $b_i$ of $ k G_i $   in $ \cE_{\ell}(G_i, s_i)$ and a character $\chi_i  \in  \cE(G_i, s_i)  \cap  \Irr(b_i)$. By Lemma~\ref{lem:lusztiggalois}, $\,^{\hat \sigma^m}\! \chi \in   \cE(G_i, s_i^{\ell^m})$ for $ m \in {\mathbb N}$. Hence by Lemma~\ref{lem:charactergalois}, ${ \sigma^m }(b_i) $ is in $   \cE_{\ell}(G_i, s_i^{\ell^m})$. It follows that if $b_i = { \sigma^m }(b_i)$ then $m$ is divisible by $d_i$. Hence the smallest subfield $k_0$ of $k$ such that $b_i \in k_0G$ contains at least $p^{d_i}$ elements. In particular, there exists no finite subfield $k_0$ of $k$ such that $b_i \in  k_0G $ for all $i$. Since the $\ell$-part of the order of the $G_i$'s is equal, this also shows  that there is no finite subfield $k_0$ of $k$ such that $b$ belongs to $k_0 G$ for all blocks $b$ of $kG$ with a given defect.
\end{exa} 

\section{Applications of the Bonnaf\'{e}-Dat-Rouquier theorem}\label{sec:BDR}

We keep all the notation of the sections before Example~\ref{ex:idempotentcoefficients}. In particular, $\bG $ is a connected reductive algebraic group defined over $\overline{\FF}_p$ with  Frobenius endomorphism $F: \bG \rightarrow \bG$.  We  fix a finite extension ${\mathbb Q}_{\ell} \subseteq   K'\subseteq \bar K$ of ${\mathbb Q}_{\ell}$ in $\bar K$ such that $K'$
is a splitting field for all sections of $\bG^F$ and of $\bG^{*F}$ and let $ \cO'$ be the ring of integers of $K'$ over ${\mathbb Z}_{\ell}$.   

The proof of our main theorem relies in a fundamental way on the reduction theorem of \cite[Section 7]{B/D/R} which we now describe.  Let $s \in \bG^{*F}$ be a semisimple $\ell'$-element. Let $\bL^* = C_{\bG^*}\left(Z^{\circ}\left( C^{\circ}_{\bG^*}(s) \right)\right)$ be the minimal Levi subgroup of $\bG^*$ containing $C^{\circ}_{\bG^*}(s)$ and let $\bL$ be a Levi subgroup  of $\bG$  dual to $\bL^*$ as in \cite[Chapter 13]{D/M}. Since $s \in \bG^{*F}$, $\bL^*$ and $\bL$ are $F$-stable. Further, the duality between $\bG$ and $\bG^*$ induces an $F$-equivariant isomorphism $N_{\bG} (\bL)  /\bL \to   N_{\bG^*} (\bL^*)/\bL^* $, ($w \mapsto w^*$) and a  bijection $\nabla( \bL, F)/\bL^F \to \nabla^*(\bL^*,  F)/\bL^{*F}  $ which is equivariant with respect to the action of $(N_{\bG} (\bL)  /\bL)^F $ on $\nabla( \bL, F)/\bL^F $ and  the action of $ (N_{\bG^*} (\bL^*)/\bL^*)^F $ on $\nabla^*(\bL^*,  F)/\bL^{*F} $: for all   $w\in (N_{\bG}(\bL)/\bL )^F $,  $(\bT, \theta)  \in  \nabla(\bL, F)$, and $  (\bT^*, s) \in \nabla^*(\bL^*, F) $, 
\begin{equation} \label{eq:dualLevi} (\bT, \theta) \stackrel{\bL}{\leftrightarrow}  (\bT^*,  s) \iff   \,^w(\bT, \theta) \stackrel{\bL}{\leftrightarrow}  \,^{w^*}\!(\bT^*,  s). \end{equation} 

Note that we identify $ \left( N_{\bG}(\bL)/\bL \right)^F $ and $N_{\bG^F}(\bL)/ \bL^F $. Let $N^* = C_{\bG^*}(s)^F. \bL^*$ and define $N$ to be the subgroup of $N_{\bG}(\bL)$ containing $\bL$ such that $N / \bL$ corresponds to $N^* / \bL^*$ via the   above  isomorphism between $N_{\bG^*}(\bL^*)/\bL^*$ and $N_{\bG}(\bL)/\bL$. As discussed in \cite[Section 7A]{B/D/R}, $N^F / \bL^F \cong N^{*F}/ \bL^{*F} \subseteq C_{\bG^*}(s)^F/ C_{\bG^*}^{\circ}(s)^F$. 
Let $e_s ^{\bG^F}$ (respectively $e_S^{\bL^F} $) denote the sum of all blocks  of $\cO'\bG^F$ (respectively $\cO' \bL^F $) in $\cE_{\ell}(\bG^F, s) $ (respectively $\cE_{\ell}(\bL^F, s)) $. 

 Let $\bV $ be the unipotent radical of a  parabolic subgroup of $\bG$ containing $\bL$ as a Levi subgroup and let $ {\bf Y}_{\bV} =  \{ g \bV \in  \bG/ \bV \mid g^{-1} F(g) \in \bV. F(\bf{V})  \}  $  be the corresponding Deligne-Lusztig $(\bG, \bL)$-variety    \cite[2.E]{B/D/R}. Let $G\Gamma_c(Y_{\bV},\cO')^{\textnormal{red}}$  be the complex of $\ell$-permutation $(\cO' \bG^F, \cO' \bL^F )$-bimodules defined in \cite[2.A, 2.C]{B/D/R}.

\begin{thm} \label{thm: BDR} \cite[Theorem 7.7]{B/D/R}  
	Suppose that $N^F/\bL^F $ is a cyclic group. The complex $G\Gamma_c(Y_{\bV},\cO')^{\textnormal{red}}e_s^{\bL^F}$ extends to a splendid Rickard complex $C$ of \break $(\cO' \bG^F e_s^{\bG^F},\cO'N^F e_s^{\bL^F})$-bimodules. 

	There is a bijection $b\mapsto b'$ between  blocks  of $\cO' \bG^F$ in $\cE_{\ell}(\bG^F, s)$ and blocks of $\cO' N^F $ covering blocks of  $\cO'\bL^F $ in $\cE_{\ell}(\bL^F,  s)$ determined by the rule $ bC = Cb'$. Given a block $b$ of $\cO' \bG^F$ in $\cE_{\ell}(\bG^F, s)$ the complex  $bCb' $ is a splendid Rickard complex of $(\cO'\bG^Fb, \cO' N^Fb' )$-bimodules  and  such that  $ H^{\dim ({\bf Y}_{\bV})}(bCb')$ induces a Morita equivalence  between $\cO' \bG^Fb $ and $\cO' N^F b'$.
\end{thm} 

	The above Morita equivalences descend to  unramified      discrete valuation  rings and  are compatible   with taking quotients by central $\ell$-subgroups.  Let   $k'$ be the residue field of $\cO'$ and let $\cO_0 := W(k') \subseteq \cO'$ be the ring of Witt vectors of $k'$ in $\cO'$.  
    
\begin{prop} \label{prop:modcentralBDR}   
	Keep the notation  and hypothesis of Theorem~\ref{thm: BDR}.  Let $b$ be a block of $\cO' \bG^F$ in $\cE_{\ell}(\bG^F, s)$ and  let $b'$ be the block of $\cO' N^F$ in bijection with $b$. 

	Let $Z \leq Z(\bG^F)$ be an $\ell$-group and  let  $\bar b $  (respectively $\bar b' $) be the block of $\cO'(\bG^F/Z) $  (respectively  $\cO' (N^F/Z) $) dominated by $b$ (respectively $\bar b'$). Then $\cO_0(\bG^F/Z) \bar b$ and $\cO_0 (N^F/Z) \bar b' $ are  Morita equivalent. Consequently, $ \cO(\bG^F/Z)  \bar b$ and $\cO (N^F/Z) \bar b'$ are Morita equivalent.
\end{prop} 

\begin{proof}   
	Note that $ Z \leq Z(\bL^F)$. By Theorem~\ref{thm: BDR}  and Lemma~\ref{lem:BDRdescent}, there exists an $(\cO_0\bG^Fb,  \cO_0N^Fb' )$-bimodule $M$ such that $M$ induces a Morita equivalence between $\cO_0\bG^Fb $ and  $\cO_0N^Fb' $  and such that $\cO'\otimes_{\cO}  M \cong H^{\dim ({\bf Y}_{\bV})}(bCb')$ as $(\cO'\bG^Fb, \cO'N^Fb')$-bimodules.

	For any  integer $i$, $H^i(bCb')$ is isomorphic to a summand of $ H^i(C)$ and 
	\[\Res^{\bG^F\times N^F}_{\bG^F \times \bL^F} H^i(C) =  H^i(G\Gamma_c(Y_{\bV},\cO')^{\textnormal{red}} e_s^{\bL^F})\]
	is isomorphic to a summand  of  $H^{i} (G\Gamma_c(Y_{\bV},\cO')^{\textnormal{red}})$. By   \cite[2.C]{B/D/R} 
	\[H^{i} (G\Gamma_c(Y_{\bV},\cO')^{\textnormal{red}})  \cong   H^{i} (G\Gamma_c(Y_{\bV},\cO'))   \cong H^i (R\Gamma_c (Y_{\bV},\cO')). \] 
	For any $ z \in Z$ and any $ g {\bf V} \in   \bf {Y}_{\bV} $, we have  $z.g {\bf V} = g {\bf V}. z $.  Hence, by the functoriality of $\ell$-adic cohomology with respect to finite morphisms, $ z u =uz $ for all $ u \in  H^i (R\Gamma_c (Y_{\bV},\cO')) $  and $z \in  Z$. By the above discussion it follows that $ z u =uz $ for  all  $ u \in  M$  and $z \in  Z$.   
	Now the first assertion result follows from \cite[Lemma~2.7]{C/K/K/S}. The second is immediate from the first as $ \cO_0  \subseteq \cO $.
\end{proof}

Thus the problem of bounding the Morita Frobenius number of $b$  is reduced to  the problem of  bounding the Morita Frobenius number of  $b'$. This in turn can be bounded by appealing to the rationality of unipotent character values. The two theorems below give instances of this  philosophy in action.

Let   $i : \bG \hookrightarrow \widetilde \bG$  be  a regular embedding as in Section ~\ref{subsec:alggroups} with dual $i^*:  \widetilde \bG^* \to  \bG^*$. Let $\widetilde \bL = Z(\widetilde \bG) \bL$ be a Levi subgroup of $\widetilde \bG$ and let $\widetilde \bL^*$ be  the  full inverse image  of $\bL^*$  under $i^*$. Then $Z(\widetilde \bL)$ is connected by \cite[Corollaire 4.4]{B} and $[\widetilde \bL , \widetilde \bL] = [\bL, \bL]$, hence the restriction of $i $ to $\bL$ is a regular embedding of $\bL $ into $\widetilde \bL$ with dual  $i^*: \widetilde \bL^* \rightarrow \bL^*$.
  
\begin{lem}\label{lem:tildeLactsonN}   
	The group $\widetilde \bL^F$ normalises $N^F$.
\end{lem}

\begin{proof}  
	Since $\widetilde \bL = Z(\widetilde \bG) \bL$ and $N \subseteq N_{\bG}(\bL)$, it is clear that $[\widetilde \bL, N] \subseteq \bL$ and hence $[ \widetilde \bL^F, N^F] \subseteq \bL^F \subseteq N^F$. 
\end{proof}

Recall the number $r_H(h)$ as introduced in Definition~\ref{defi:rH}.

\begin{thm}\label{prop:BDRresult} 
	Keep the notation and hypothesis of Proposition~\ref{prop:modcentralBDR}. Let  $r = r_{ \bL^{*F}}( s)$  and suppose that $C_{ \bL^*}^{\circ}( s)$ has all classical components. Then   
	\[\cO ( N^F  /Z ) \hat \sigma^r(\bar {b'})  \cong \cO (N^F/Z) \bar{b'}. \]
\end{thm}	

\begin{proof}  
	Let $\tilde s$ be a semisimple $\ell'$-element of $\widetilde \bL^{*F}$ with $i^*(\tilde s) = s $. By Lemma~\ref{lem:rgood} $\tilde t:=\tilde s^{\ell^r-1} \in  Z(\widetilde \bL^{*})^F$. Let $\tilde \tau$ be the linear character of $\widetilde \bL^F$ corresponding to $\tilde t $  as in Lemma~\ref{lem:duality} (ii) applied to $\widetilde \bL$. By Lemma~\ref{lem:thetatildeinG}, $^{\hat \sigma ^r} \! \tilde \theta  = \tilde \theta  \tilde {\tau}_{\widetilde \bT^F}$ for all $(\widetilde \bT, \tilde \theta) \in \nabla (\widetilde \bL, F, \tilde s)$. 
	
	Let $c$ be a  block of $\cO \bL^F$ covered by $b'$ and	let $\chi \in \textnormal{Irr}(c) \cap \cE(\bL^F, s)$. Then since $C_{\bL^*}^{\circ}(s)$ has all classical components, it follows from Lemma~\ref{lem:conjofchi}, applied to $\bL$,  that there exists an $x \in \widetilde \bL^F$ such that 
	\[^{\hat \sigma ^r} \! \chi = \, ^x(\chi \tilde \tau_{\bL^F}).\]
	
	Let $t := s^{\ell^r -1}= i^*(\tilde t) $. Since $i^*$ is a surjective homomorphism and $\tilde t$ is central, $ t \in Z(\bL^*)$.  Further, by Lemma~\ref{lem:torusembed},     $\tau := \tilde \tau_{\bL^F}$ is the linear character of $\bL^F$ corresponding to $t$. Since $N^*$ centralises  $s$, $N^*$ centralises  $t$. Hence by Lemma~\ref{lem:duality} and by (\ref{eq:dualLevi}), $\tau$ is $N^F$-stable. Since $N^F/\bL^F$ is cyclic  and of $\ell'$-order, $\tau$ extends to an irreducible character $\check \tau$ of $N^F$  of $\ell'$-order. Note that  $\check \tau$ takes values in $\cO$.
	
	Let $\check \chi \in \textnormal{Irr}(N^F| \chi)\cap \Irr(b')$  and let $d$ be the block  of $\cO N^F $ containing ${ \check \chi} {\check \tau}$. Since $\widetilde \bL^F$ acts on $N^F$ by Lemma~\ref{lem:tildeLactsonN}, both $^{\hat \sigma ^r} \check \chi $ and ${^x \! \check \chi} {^x \!\check \tau}$ are elements of $\textnormal{Irr}(N^F | ^{\hat \sigma ^r} \chi) =  \textnormal{Irr}(N^F | {^x \! \chi} {^x \!\tau}) $. Also, $ {^x \! d}$ is the block of $\cO N^F$ containing ${ ^x \!\check \chi} {^x \!\check \tau}$. Since $N^F/\bL^F$ is cyclic and hence abelian, it follows from Lemma~\ref{lem:coveringsameblock} that there is an isomorphism $\cO N^F \hat \sigma^r(b') \cong  \cO N^F {^x \! d} $ which restricts to the identity on $Z \leq Z(\bL^F)$. Conjugation by $ x^{-1}$ induces an isomorphism $\cO N^F {^x \! d}  \cong  \cO N^F { \! d}$ which is the identity  on $Z \leq \bG^F$.  Finally the $\cO $-algebra automorphism $\cO N^F \to \cO N^F $ satisfying $n \mapsto  \check \tau  (n^{-1}) n  $ for all  $ n\in \bN^F $, restricts to an isomorphism  $ \cO N^F { \! d} \cong  \cO N^F b'$. This last isomorphism  also restricts to the identity on $Z$ since $\check \tau $ is an $\ell'$-character,  and $Z$ is an $\ell$-group. Thus  $\cO N^F \hat \sigma^r( b') \cong \cO N^F b'$ via an isomorphism which is the identity on $Z$. 

	Hence $\cO ( N^F  /Z ) \overline{ \hat \sigma^r(b') } \cong \cO (N^F/Z) \bar b'$ where $\bar b'$ denotes the block of $\cO (N^F/Z)$ dominated by $b'$. The result then follows by Lemma~\ref{lem:domblocks} (iv).
\end{proof}

The next  result  is   similar   in spirit  to the previous one  and will  be used to deal with some cases where $C_{ \bL^*}^{\circ}( s)$ has an exceptional component.  

\begin{thm}\label{prop:central} 
	Keep the notation and hypothesis of  Proposition~\ref{prop:modcentralBDR}. Suppose that $s \in Z(\bL^{*})^F$ and $a \in {\mathbb N} $  is  such that $^{\hat \sigma ^a}\! \chi' = \chi'$ for all unipotent characters $\chi' \in \textnormal{Irr}(\bL^{F})$.   Then   
	\[\cO ( N^F  /Z ) \hat \sigma^a(\bar {b'})  \cong \cO (N^F/Z) \bar{b'}.\]
\end{thm}

\begin{proof}
	Let $\psi \in \textnormal{Irr}(b') $ and $ \chi\in  \cE(\bL^F, s)$  such that $ \psi $ covers $\chi $. Let $\hat s$ denote the linear character of $\bL^F$ corresponding to $s$ as in Lemma~\ref{lem:duality} (ii). Note that $\hat s$ is $N^F$-stable. By \cite[Proposition 13.30, (ii)]{D/M} there exists a unipotent character $\chi'$ of $\bL^F$ (necessarily unique) such that $\chi = \hat s \chi'$. By assumption, therefore $^{\hat \sigma ^a}\! \chi = \, ^{\hat \sigma ^a}\! (\hat s \chi') = \, (^{\hat \sigma ^a}\! \hat s) \chi' $. Let $\xi = {^{{\hat \sigma}^a}\!{(\hat{s})}} \hat{s}^{-1}$. Then $^{\hat \sigma ^a}\! \chi = \xi \chi$. 
	
	Since $\xi$ is a linear $N^F$-stable character of $\bL^F$ and $N^F/\bL^F $ is cyclic by assumption, $\xi$ extends to a linear character $\hat{\xi}\in $ Irr$\left(N^F\right)$. Hence $^{\hat \sigma ^a} \!\psi$ and $\hat \xi \psi$ are both elements of $\textnormal{Irr}(N^F | \, ^{\hat \sigma ^a}\chi) = \textnormal{Irr}(N^F | \xi \chi)$. Let $d$ be the block of $\cO N^F$ containing $\hat \xi \psi$.   Since $N^F/\bL^F$ is abelian, it follows from  Lemma~\ref{lem:coveringsameblock} that there is an isomorphism $\cO N^F \hat \sigma^a(b')\cong \cO N^F {d} $   which restricts to the identity on $Z \leq Z(\bL^F)$. As in Theorem~\ref{prop:BDRresult}, there is an isomorphism $\cO N^F d \cong  \cO N^F b'$ which restricts to the identity on $Z \leq Z(\bL^F)$. Composition induces an isomorphism $\cO N^F \hat \sigma^a(b') \cong \cO N^F b'$ which restricts to the identity on   $Z \leq Z(\bL^F)$. 
\end{proof}

\begin{rem} \label{rem:reesuzuki} Here and in Section \ref{sec:BDR}, we have assumed that $F$ is a Frobenius morphism. However, suitable analogues of the results of these sections, in particular  Theorem~\ref{thm: BDR}, hold under the weaker assumption that    some power of $F$ is a Frobenius morphism, and thus may also be applied to the Ree and Suzuki groups. We will make use of this  in Proposition~\ref{prop:suzukiree}. We have chosen to stick to the Frobenius case for the general exposition as most of our references for these sections make this assumption. 
\end{rem}
\section{Blocks of finite groups of Lie type in non-defining characteristic}\label{sec:FGLT}
We  keep the  notation of Section~\ref{sec:BDR}.  In addition we  assume in this section that $\bG$ is simple and simply-connected.  Throughout $b$  will denote a  block of $\cO \bG^F$ in $\cE_{\ell} (\bG^F, s)$, $Z$ an $\ell$-subgroup of $ Z(\bG^F)$ and $\bar b $   the block of $\cO (\bG^F/Z)$ dominated by $b$. Further, whenever $N^F/ \bL^F$ is cyclic we will denote by $b'$ the block of $\cO N^F$  in bijection with $b$  as in Theorem~\ref{thm: BDR}   and  $ \bar b' $  will denote the block of $\cO (N^F/ Z) $ dominated by $ b'$.

Our  first   result  lifts  some of the results  of \cite{F} on unipotent blocks  to $\cO$. We note that the  character theoretic arguments applied to blocks of $k \bG^F$ in \cite{F} also apply to blocks of $\cO \bG^F$. In particular, if $b$ is a block of $\cO \bG^F$ containing characters whose sum is rational valued (that is, there exist $\chi_1, \dots, \chi_r \in \textnormal{Irr}(b)$  such that $\left(\chi_1 + \dots + \chi_r\right) (g) \in \mathbb{Q}$ for all $g \in G$), then $\hat \sigma (b) = b$. 

\begin{prop}
	\label{prop:unipotent}
 Suppose that $s =1$.   Then $\cO (\bG^F/Z) \hat \sigma^r (\bar b)  \cong \cO (\bG^F / Z) \bar b$ with $r = 1,2$ and if $\bG$ is not of type $E_7$ or $E_8$, $\cO (\bG^F/Z) \hat \sigma (\bar b)  \cong \cO (\bG^F / Z) \bar b$.
\end{prop}

\begin{proof}
	The unipotent characters of classical groups are determined by their  uniform projections \cite[8.1(a)]{Lu88} and  the characters $R_{\bT}^{\bG} (1)$ are rational valued, hence the unipotent characters of  classical groups  are  rational valued. So if $\bG$ is of classical type then $\hat \sigma(b) = b$. Assume now that $\bG$ is of exceptional type, and let $b = b_{\bG^F}(\bL, \lambda)$ for a unipotent $e$-cuspidal pair $(\bL, \lambda)$ of $\ell$-central defect (see \cite[Theorem 4.4]{C/E} and \cite[Th\'{e}or\`{e}me A]{E}). If $\bG = \bL$ then $b$ has cyclic defect groups (see the proof of \cite[Theorem 5.5]{F}), so by Proposition~\ref{prop:cyclicklein} we  may  assume that $\bL$ is a proper Levi subgroup of $\bG$. 
	
	If $\bG$ is of type $E_7$ or $E_8$ then by \cite[Table 1]{G}, $^{\hat \sigma^r}\lambda = \lambda$ for some $r \leq 2$. Hence by the proof of \cite[Lemma 5.2]{F}, $\hat \sigma^r (b) = b $ and therefore $\cO \bG^F \hat \sigma^r (b) \cong \cO \bG^F b$ via an isomorphism which is trivial on $Z$. Thus by Lemma~\ref{lem:domblocks} (iv), $\cO (\bG^F / Z) \hat \sigma^r(\bar b) \cong \cO (\bG^F/Z) \bar b$ for some $r \leq 2$. If $\bG$ is not of type $E_7$ or $E_8$ then $^{\hat \sigma}\lambda = \lambda$ by \cite[Table 1]{G} and hence, by the same arguments, $\cO (\bG^F / Z) \hat \sigma(\bar b) \cong \cO (\bG^F/Z) \bar b$. 
\end{proof}

Recall that by \cite[Lemma 13.14 (iii)]{D/M}, $C_{\bG^*}(s)/ C_{\bG^*}^{\circ}(s)$ is isomorphic to a subgroup of Irr$(Z(\bG)/Z^{\circ}(\bG))  \cong \Irr(Z(\bG)) \cong   Z(\bG) $.   Further, $Z(\bG)$   is cyclic in all cases except possibly when $\bG$ is of type $D_m$ with $m$ even (see \cite[Table 24.2]{M/T}). Therefore $C_{\bG^*}(s)^F/ C_{\bG^*}^{\circ}(s)^F$, and hence $ N^F/ \bL^F$, is cyclic in all cases except possibly when $\bG$ is of type $D_m$ with $m$ even. We will use this fact without further comment. Also, we note that in many cases below $Z=1 $. 

Recall that a semisimple element $t$ of a connected reductive group $\bH$ is called isolated if its connected centraliser $C_{\bH}^{\circ}(t)$ is not contained in any proper Levi subgroup of $\bH$. The element $t$ is isolated in $\bH$  if and only if the  image, $\bar t$, of $t$ in $\bH/Z(\bH)$ is isolated in $\bH/Z(\bH) $. The quotient $ \bH/Z(\bH) $ is a direct product of adjoint simple groups and $\bar t$ is isolated in $\bH/Z(\bH)$ if and only if the projection of $\bar t$ on  every  simple  factor of $\bH /Z(\bH)$ is  isolated. Suppose that $F': \bH \rightarrow \bH$ is a Frobenius morphism. Then if $ t\in \bH^{F'}$, $a_{\bH^{F'} }(t)$ is a divisor of the order of $ \bar t$ in $\bH/Z(\bH) $. Thus, by the classification of isolated elements  in simple algebraic groups, \cite[Section 5]{B2} (see also  \cite[Table 6.2]{Ta}), if $t \in \bH^{F'}$ is an isolated  $\ell'$-element of $\bH$ and $\bH$ has all  classical components, then $a_{\bH^{F'}} (t)$ equals $1$ or $2$ and if $\bH$ has at most one component not of type A, then $a_{\bH^{F'}}(t)\leq 6$. We note that Table 2 of \cite{B2} lists an isolated element of order $4$ in type  $D$, but the listed element is in fact of order $2$.

By choice of $\bL^*$, $s$ is isolated in $\bL^*$ and  whenever    $s$ is  isolated in $\bG^*$, then $N^* =\bL^* = \bG^*$ and $b =b'$.

\begin{prop}\label{prop:ABC}
	Suppose that $\bG$ is of type A, B or C. Then $\cO (\bG^F/Z) \bar b $  is Morita equivalent to $ \cO (N^F/Z) \bar b' $ and $ \cO (N^F/Z) \bar b' \cong \cO (N^F/Z) \hat\sigma (\bar b')$. 
\end{prop}

\begin{proof}
	Since $\bL^*$ is a classical group, $a_{\bL^{*F}}(s) \leq 2$ for all isolated elements $s \in \bL^{*F}$. Since $s$ is an $\ell'$-element, $a_{\bL^{*F}}(s) = 2$ can only occur when $\ell$ is odd. Hence $r_{\bL^{*F}}(s) = 1$. 	Thus the result follows from Proposition~\ref{prop:modcentralBDR} and Theorem~\ref{prop:BDRresult}. 
\end{proof}

\begin{prop}\label{prop:D}
	Suppose that $\bG$ is of type $D$.
	\begin{enumerate}[(i)]
		\item  If $N^F/\bL^F$ is cyclic, then   $\cO (\bG^F/Z) \bar b $  is Morita equivalent to  $ \cO (N^F/Z) \bar b' $   and $ \cO (N^F/Z) \bar b'  \cong     \cO (N^F/Z)   \hat\sigma (\bar b')$. 
		\item  If  $N^F/\bL^F$ is not cyclic, then $\cO (\bG^F/Z) \hat \sigma^r\! (\bar b)  \cong \cO (\bG^F / Z) \bar b$  for some $ r=1, 2$.
	\end{enumerate}
\end{prop}

\begin{proof} 
	The proof of part (i) is identical to that  of  Proposition~\ref{prop:ABC}. Suppose that $N^F/\bL^F$ is  not cyclic. Then  $N^F/\bL^F \cong C_2 \times C_2$.   Further, $\bG^F = \textnormal{Spin}_{2n}^+(q)$ and $s$ is a quasi-isolated element of $\bG^{*F}$ of order 4 such that $C_{\bG^*}^{\circ}(s)$ is of type $A_{2n-3}$ (see \cite[Remark 9.24]{Ca2}). Since $N^F / \bL^F \subseteq \left( C_{\bG^*}(s) / C^{\circ}_{\bG^*}(s) \right)^F$ and $\textnormal{exp}\left(C_{\bG^*}(s) / C^{\circ}_{\bG^*}(s)\right) = 2$ divides the order of the $\ell'$-element $s$ \cite[Remark 13.15 (i)]{D/M}, $\ell$ is odd. Let $r = r_{\bG^{*F}}(s)$ and note that $r \leq 2$. 

	Let $i:\bG \hookrightarrow \widetilde \bG$ be a regular embedding and $\tilde s \in \widetilde \bG^{*F}$ a semisimple $\ell'$-element such that $i^*(\tilde s) = s$. It follows from Lemma~\ref{lem:thetatildeinG} (ii) that $^{\hat \sigma^r}{\tilde \theta} = \tilde \theta \tilde \tau_{\widetilde \bT^F}$ for all $(\widetilde \bT, \tilde \theta) \in \nabla(\widetilde \bG, F, \tilde s)$, where $\tilde \tau$ is the linear character of $\widetilde \bG^F$ corresponding to $\tilde s^{\ell^r - 1}$. 

	Suppose that $\chi \in \textnormal{Irr}(b) \cap \cE(\bG^F, s)$. 
	By Lemma~\ref{lem:conjofchi}, since $\bG$ is of classical type, ${^{\hat\sigma^r} }\chi = \, ^x(\chi \, \tilde \tau_{\bG^F})$ for some $x \in \widetilde \bG^F$. Let $d$ be the block of $\cO \bG^F$ containing $\chi \, \tilde \tau_{\bG^F}$. Then $\hat \sigma^r(b)$ is the block of $\cO \bG^F$ containing $^x(\chi \, \tilde \tau_{\bG^F})$, so $\hat \sigma^r(b) = \, ^xd$. Conjugation by $x^{-1}$ induces an isomorphism $\cO \bG^F \, ^x d \cong \cO \bG^F d$ which is the identity on $Z \leq \bG^F$. The automorphism $\cO \bG^F \rightarrow \cO \bG^F$ given by $g \mapsto \tilde \tau_{\bG^F} (g^{-1}) g$ induces an isomorphism $\cO \bG^F d \cong \cO \bG^F b$ which is also trivial on $Z$. Hence, $\cO \bG^F \hat \sigma^r(b) \cong \cO \bG^F b$ via an isomorphism which is trivial when restricted to $Z$  yielding an isomorphism   $\cO (\bG^F/Z) \hat \sigma^r\! (\bar b)  \cong \cO (\bG^F / Z) \bar b$.  
\end{proof}

\begin{prop}\label{prop:G2F4E6}
	Suppose that $\bG$ is of type $G_2$, $F_4$ or $E_6$  and $s\ne 1 $.  Then  $\cO (\bG^F/Z) \bar b $ is Morita equivalent to  $ \cO (N^F/Z) \bar b' $. If $s$ is isolated in $\bG^*$ and either $o(s) = 3$ and $\ell \equiv 2$ mod $3$, or $o(s) = 4$ and $\ell \equiv 3$ mod $4$,  then  $ \cO (N^F/Z) \bar b'  \cong     \cO (N^F/Z)   \hat\sigma^r\! (\bar b')$  for some $ r=1, 2$.  Otherwise,  $ \cO (N^F/Z) \bar b'  \cong     \cO (N^F/Z)   \hat\sigma\! (\bar b')$.   
\end{prop}

\begin{proof} 
	The first assertion is Proposition~\ref{prop:modcentralBDR}. If $s$ is isolated in $\bG^*$ and either $o(s) = 3$ and $\ell \equiv 2$ mod $3$, or $o(s) = 4$ and $\ell \equiv 3$ mod $4$, then $r_{\bL^{*F}}(s) = 2$. Otherwise  $r_{\bL^{*F}}(s) = 1$. Since for all non-trivial semisimple $\ell'$-elements $s \in \bG^{*F}$, $C^{\circ}_{\bL^*}( s)$ has all classical components by \cite[Section 5]{B2}, 	the result follows from Theorem~\ref{prop:BDRresult}. 
\end{proof}

\begin{prop}\label{prop:E7}
	Suppose that $\bG$ is of type $E_7$ and $s\ne 1 $. Then  $\cO (\bG^F/Z) \bar b $  is Morita equivalent to  $ \cO (N^F/Z) \bar b' $. If $\ell \equiv 2$ mod $3$ and one of the following holds, 
	\begin{itemize}
		\item $s$ is isolated in $\bG^{*}$ and $o(s) = 3$
		\item $s$ is not isolated in $\bG^{*}$, $s \in Z(\bL^*)^F$ and $\bL^*$ has a component of type $E_6$
		\item $s$ is not isolated in $\bG^{*}$, $\bL^*$ has a component of type $E_6$ and $a_{\bL^*}(s) = 3$,
	\end{itemize}
	or if $\ell \equiv 3 $ mod $4$ and $s$ is isolated in $\bG^{*}$ with $o(s) = 4$, then  $ \cO (N^F/Z) \bar b'  \cong     \cO (N^F/Z)   \hat\sigma^r\! (\bar b')$  for some $ r=1,2 $.  Otherwise, $ \cO (N^F/Z) \bar b'  \cong     \cO (N^F/Z)   \hat\sigma\! (\bar b')$.     
\end{prop}

\begin{proof} 
	The first assertion is just Proposition~\ref{prop:modcentralBDR}. Suppose that $s$ is isolated in $\bG^*$. Then if $o(s) = 3$ and $\ell \equiv 2 $ mod $3$ or if $o(s) = 4$ and $\ell \equiv 3 $ mod $4$, $r_{\bG^{*F}}(s) = 2$. Otherwise $r_{\bG^{*F}}(s) = 1$. Since $C_{\bG^*}^{\circ}(s)$ has all classical components by \cite[Section 5]{B2}, the result follows by applying Theorem~\ref{prop:BDRresult} with $\bL^* = N^* = \bG^*$. 

	Now suppose that $s$ is not isolated in $\bG^*$ and $s \in Z(\bL^*)^F$. If $\ell \equiv 2 $ mod $3$ and $\bL^*$ has a component of type $E_6$, then the minimal $a$ such that $^{\hat \sigma ^a} \chi' = \chi'$ for all unipotent characters $\chi'$ of $\textnormal{Irr}(\bL^F)$ is 2. Otherwise the minimal such $a$ is 1, so the result follows from Theorem~\ref{prop:central}.   

	Finally, suppose that $s$ is not isolated in $\bG^*$ and $s \notin Z(\bL^*)^F$. If $\ell \equiv 2 $ mod $3$, $\bL^*$ has a component of type $E_6$ and $a_{\bL^*}(s) = 3$, then $r_{\bL^{*F}}(s) = 2$. Otherwise, $r_{\bL^{*F}}(s) = 1$. Again, since $C_{\bL^*}^{\circ}(s)$ has all classical components by \cite[Section 5]{B2}, the result follows from Theorem~\ref{prop:BDRresult}. 
\end{proof}

There are three non-trivial isolated elements in $E_8$ whose centralisers have an exceptional component. These cases are dealt with separately in Section~\ref{subsec:E8}. 

\begin{prop}\label{prop:E81}
	Suppose that $\bG$ is of type $E_8$, $ s\ne 1 $ and that if $s$ is isolated in $\bG^*$,  then  $C_{\bG^*}(s)$ has all classical components. The block $\cO (\bG^F/Z) \bar b $  is Morita equivalent to $ \cO (N^F/Z) \bar b' $.    
	\begin{itemize} 
		\item If $s$ is isolated in $\bG^*$, $o(s) = 5$ and $\ell \equiv 2 $ or $ 3 $ mod $5$ then   $ \cO (N^F/Z) \bar b'  \cong     \cO (N^F/Z)   \hat\sigma^r\! (\bar b')$  for some $ r\leq 4 $. \vspace{1ex}  
		\item If one of the following holds,
		\begin{itemize}
			\item $s$ is isolated in $\bG^*$, $o(s) = 3$ and $\ell \equiv 2 $ mod $3$; or $o(s) = 4$ and $\ell \equiv 3 $ mod $4$; or $o(s) = 5$ and $\ell \equiv 4 $ mod $5$; or $o(s) = 6$ and $\ell \equiv 5$ mod $6$, \vspace{.5ex}
			\item $s$ is not isolated in $\bG^*$, $s \in Z(\bL^*)^F$, $\bL^*$ has a component of type $E_6$ or $E_7$ and $\ell \equiv 2$ mod $3$, \vspace{.5ex}
			\item $s$ is not isolated in $\bG^*$, $s \notin Z(\bL^*)^F$, $\bL^*$ has a component of type $E_7$, $a_{\bL^*}(s) = 4$ and $\ell \equiv 3$ mod $4$, or \vspace{.5ex}
			\item $s$ is not isolated in $\bG^*$, $s \notin Z(\bL^*)^F$, $\bL^*$ has a component of type $E_6$ or $E_7$, $a_{\bL^*}(s) = 3$ and $\ell \equiv 2$ mod $3$, \vspace{.5ex}
		\end{itemize}
		then  $ \cO (N^F/Z) \bar b'  \cong     \cO (N^F/Z)   \hat\sigma^r\! (\bar b')$  for some $ r=1,2 $.    \vspace{1ex} 
		\item  In all other cases $ \cO (N^F/Z) \bar b'  \cong     \cO (N^F/Z)   \hat\sigma\! (\bar b')$.    
	\end{itemize} 
\end{prop}

\begin{proof}
	First suppose that $s$ is isolated in $\bG^*$. If $o(s) = 5$ and $\ell \equiv 2 $ or $ 3 $ mod $5$ then $r_{\bG^{*F}}(s) = 4$; if $o(s) = 3$ and $\ell \equiv 2 $ mod $3$, or $o(s) = 4$ and $\ell \equiv 3 $ mod $4$, or $o(s) = 5$ and $\ell \equiv 4 $ mod $5$, or $o(s) = 6$ and $\ell \equiv 5$ mod $6$ then $r_{\bG^{*F}}(s) = 2$; otherwise $r_{\bG^{*F}}(s) =1 $. As we are assuming that $C_{\bG^*}(s)$ has all classical components, we can then apply Theorem~\ref{prop:BDRresult}.

	Now suppose that $s$ is not isolated in $\bG^*$ and $s \in Z(\bL^*)^F$. If $\ell \equiv 2 $ mod $3$ and $\bL^*$ has a component of type $E_6$ or $E_7$, then the minimal $a$ such that $^{\hat \sigma ^a} \chi' = \chi'$ for all unipotent characters $\chi'$ of $\textnormal{Irr}(\bL^F)$ is 2. Otherwise the minimal such $a$ is 1, and the result follows from Theorem~\ref{prop:central}. 

	Finally, suppose that $s$ is not isolated in $\bG^*$ and $s \notin Z(\bL^*)^F$. If $\bL^*$ has a component of type $E_7$, $\ell \equiv 3 $ mod $ 4$ and $a_{\bL^*}(s)= 4$ or if $\bL^*$ has a component of type $E_6$ or $E_7$, $\ell \equiv 2 $ mod $3$ and $a_{\bL^*}(s) = 3$, then $r_{\bL^{*F}}(s) = 2$. Otherwise, $r_{\bL^{*F}}(s) = 1$. 	Thus since $C_{\bL^*}^{\circ}(s)$ has all classical components, the result follows from Theorem~\ref{prop:BDRresult}.
\end{proof}

\subsection{Isolated blocks of $E_8$ with $C_{\bG^*}(s)$ of non-classical type}\label{subsec:E8}
In this subsection we assume that  $\bG = E_8$.  We will deal with the non-trivial isolated semisimple elements $s \in \bG^{*F}$ which are not covered by Proposition~\ref{prop:E81}.

\begin{Notation}\label{not:mathcalCB}  We let $\mathcal{C}_b$ denote the set of blocks in $\cE_{\ell}(\bG^F, s)$ which are Galois conjugates of $b$, that is, blocks of the form ${\hat \sigma ^m} (b) $, $ m \in {\mathbb N} $.  
\end{Notation}

\begin{lem}
	\label{lem:orders}
Set $|\mathcal{C}_b| = m$ and $r = r_{\bG^{*F}}(s)$. Then $\hat \sigma^{r}(b)  $ is a block in $ \cE_{\ell}\left(\bG^F, s\right)$ and $ b = \hat \sigma ^{n} (b)$ for some $n \leq  rm $. 
\end{lem}

\begin{proof}  
	By Lemma~\ref{lem:lusztiggalois} (ii) the  action of  the group $\langle \hat \sigma  \rangle  $ on $\Irr (\bG^F)$  induces an   action on the set of Lusztig series $\cE (\bG^F, t)$, where $t$ runs over the $\bG^{*F}$-conjugacy classes of $\ell'$-elements  of $\bG^{*F}$. Further, since $Z(\bG^{*F})=1 $, $a_{\bG^{*F}} (s)  =o(s)$ and  it follows that  $s^{\ell^r}= s$. Thus   the stabiliser in  $\langle \hat \sigma  \rangle$  of $\cE(\bG^F,  s)$ is of the form $ \langle \hat \sigma ^u \rangle  $, for some  non-negative integer $u$ dividing $r$. In particular, the first assertion is proved. Now, the set of $\ell$-blocks of $\bG^F$  contained in $\cE_{\ell} (\bG^F, s)$ is $ \langle \hat \sigma ^u \rangle  $ invariant and ${\mathcal  C}_b $ is the $\langle \hat \sigma ^u \rangle  $-orbit of $b$ under the action of $ \langle \hat \sigma ^u \rangle $ on this set, hence $ b = \hat \sigma ^{um} (b) $, proving the second assertion.
\end{proof}

\begin{prop}\label{prop:E82}
		\label{prop:isolatedE8}
		Let $\bG$ be a simple algebraic group of type $E_8$  and suppose that   $1\ne s \in \bG^{*F}$  is  an isolated semisimple $\ell'$-element such that $C_{\bG^*}(s)$ has an exceptional component. Then $o(s)=2 $ or $o(s)= 3$. If $o(s) = 3$ then $\cO \bG^F b  \cong  \cO \bG^F \hat \sigma ^r (b)  $  for some $ r \leq  4 $ and if $o(s) = 2$ then $\cO \bG^F b  \cong  \cO \bG^F \hat \sigma ^r (b)$ for some $ r\leq 2 $. 
\end{prop}

\begin{proof} 
	By \cite[Table 1, Table 5]{K/M}, it is enough to consider $s$ in the following cases.
	
	\break
	\begin{longtable}{c|c}
		$o(s)$ 	& 	Components of $C_{\bG^*}(s)^F$\\ \hline\hline
		2		& 	$E_7 \times A_1$ \\
		3		&	$E_6 \times A_2$ \\
		3		&	${^2E}_6 \times {^2A}_2$
	\end{longtable}	
In particular, the first assertion of the proposition holds. If $b$ has cyclic defect groups, then the remaining assertion follows from Proposition \ref{prop:cyclicklein}. For  the rest of the proof we assume that $b$ has non-cyclic defect groups. If $ o(s)=2 $, then $r_{\bG^{*F}} (s) =1 $ and  if $o(s)=3$, then $r_{\bG^{*F}} (s) =2$. Thus, by Lemma~\ref{lem:orders}  it suffices to prove that  $ |\mathcal{C}_b| \leq 2$. Let $\bG^F = E_8(q)$, and if   $\ell $ is odd, let  $e$ be the order of $q$  modulo $\ell $. If  $ \ell =2 $ let $ e$ be the order of $ q$ modulo $4$. We will use the parametrisation of $\ell$-blocks of $\bG^F$  by $e$-cuspidal pairs   to obtain the desired bound on $ |\mathcal{C}_b|$.

First suppose that $\ell \leq 5$ (i.e. $\ell$ is bad for $\bG$).   By \cite[Theorem 1.2]{K/M}, the blocks in $\cE_{\ell}(\bG^F, s)$ are in bijection with $\bG^F$-conjugacy classes of pairs $(\bM, \la)$ of $\bG$ such that  $(\bM, \la) $ is an $e$-cuspidal pair of $\bG$,  and  $\la  \in \cE(\bM^F, s)$ is of quasi-central $\ell$-defect. The bijection is described via Lusztig induction: a block corresponds to a pair $ (\bM, \la)$ if and only if all irreducible constituents of $R_{\bM}^{\bG} (\la)$ lie in the block. The tables in Section 6 of \cite{K/M} list the $e$-cuspidal pairs of $\bG$. Since $\hat \sigma $ commutes with Lusztig induction, if the pair $(\bM, \la)$   corresponds to the  block $b$, then $(\bM,  \,^{\hat\sigma^m }\la)$  corresponds to $\hat\sigma^m(b)$ for $\hat  \sigma^m(b)$ in   $\mathcal{C}_b$. Further, if $g\in \bG^F$, then $^g \lambda = \lambda$ if and only if $^g(\,^{\hat \sigma^m} \! \lambda ) =\,  ^{\hat \sigma ^m} \! \lambda$, hence $(\bM, \la )$ and $ (\bM, \,^{\hat \sigma^m} \! \lambda)$ have the   same relative Weyl group. Thus, all blocks in $ \mathcal{C}_b $ correspond to the same numbered line of the tables in Section 6 of \cite{K/M}. Moreover, since  $\la $ and $\,^{\hat\sigma^m }\la $ have  the same degree, by the degree formula for Jordan correspondence (see \cite[Remark 13.24]{D/M}), it follows that if $\mathcal {C}_b $ has more than one element  then  the  $\lambda $ column of the relevant line of the table contains  at least two entries of the same degree. Finally, the only relevant lines of the tables  are those which correspond to blocks with non-cyclic and in particular  non-trivial defect. Inspection of the tables yields $|\mathcal{C}_b|\leq 2$. 

Now suppose that $\ell \geq 7$.   In this case  block distributions  are described in \cite{C/E4}.  By  \cite[Theorem 4.1]{C/E4},  the blocks of $\bG^F$  in $\cE_{\ell}(\bG^F, s)$ are in bijection with $\bG^F$-conjugacy classes of pairs $(\bM, \la)$ of $\bG$ such that  $(\bM, \la)$ is an $e$-cuspidal pair of $\bG$ and with $\la  \in \cE(\bM^F, s)$; the bijection is defined by the same condition  on Lusztig induction as for bad $\ell$ above. Further, if $ (\bM, \la) $ is an $e$-cuspidal pair for $\bG$ and $\alpha \in \cE(C_{\bM^*}(t)^F, 1) $ is in Jordan correspondence with $\la$, then $(C_{\bM^*} (t)^F,  \alpha )$ is a unipotent $e$-cuspidal pair for $C_{\bG^*}(t)$  and the $\ell$-block of $\bG^F$  corresponding to  $ (\bM, \la)$ and the unipotent $\ell$-block of  $C_{\bG^*} (t)^F $  corresponding to $ (C_{\bM^*} (t)^F, \alpha)$   have  isomorphic defect groups  \cite[Theorem 4.1, Proposition 5.1]{C/E4}.  Now, if the multiplicity of the $e$th-cyclotomic polynomial $\Phi_e$ in the polynomial order of $C_{\bG^*}(s)^F$ is at most $1$, then the Sylow $\ell$-subgroups are cyclic, and  consequently every $\ell $-block of $ C_{\bG^*}(s)^F$ has cyclic defect. Thus, we may assume that this multiplicity is at least $2$; that is, $e$ is one of $ 1, 2, 3, 4, 6$.  

Suppose that $e = 1$ or $4$.   We   refer   again to the tables  of  \cite[Section 6]{K/M}  for a list of  $e$-cuspidal pairs noting that  now  unnumbered lines  also correspond to blocks.   By the  same considerations as for the bad primes case   it follows that   any two  blocks  in   $\mathcal{C}_b $  correspond to the same (numbered or unnumbered) line of the table  and that   if $  |\mathcal {C}_b| >1  $ then  the  $\lambda $  column of the relevant line  of   the  table    contains  at least two entries of the same degree  and of  positive  $\ell$-defect.  As before, we obtain $|\mathcal{C}_b| \leq 2$.
The results for $e = 2$ follow by Ennola duality from the $e= 1$ case \cite[Section 3A]{B/M/M}.

Now suppose that $e = 3$. The results for $e = 6$ will again follow by Ennola duality. The third column of the following table lists pairs $(C_{\bM^*}(s)^F, \alpha)$ where $(C_{\bM^*}(s), \alpha)$ is a unipotent $3$-cuspidal pair of $C_{\bG^*}(s)$. These were calculated using CHEVIE \cite{CHEVIE} and GAP \cite{GAP3}.  
	
\footnotesize
	\begin{longtable}{ c| c| c | c  }
		\caption{Unipotent $3$-cuspidal pairs of $C_{\bG^*}(s)$ } \\
		
		Row
		& $C_{\bG^*}(s)^F$
		& $(C_{\bM^*}(s)^F, \alpha)  $
		& Defect groups \\ \hline\hline
		
		
		$1$
		& $E_7(q)A_1(q)$
		& $\left\{ \begin{array}{c}
		(\Phi_3^3.A_1(q), 1 \otimes 1) \\
		(\Phi_3^3.A_1(q), 1 \otimes \phi_{11})
		\end{array} \right.$
		& Not cyclic		\\
		
		$2$
		& $E_7(q)A_1(q)$
		& $\left\{ \begin{array}{c}
		(\Phi_1\Phi_3 .{^3\!D}_4(q).A_1(q), {^3\!D}_4[-1] \otimes 1) \\
		(\Phi_1\Phi_3 .{^3\!D}_4(q).A_1(q), {^3\!D}_4[-1] \otimes \phi_{11})
		\end{array} \right.$
		& Cyclic
		\\

		$3$
		& $E_7(q)A_1(q)$
		& $\left\{ \begin{array}{c}
		(\Phi_3.A_5(q).A_1(q), \phi_{42} \otimes 1) \\
		(\Phi_3.A_5(q).A_1(q), \phi_{42} \otimes  \phi_{11}) \\
		(\Phi_3.A_5(q).A_1(q), \phi_{2211} \otimes 1) \\
		(\Phi_3.A_5(q).A_1(q), \phi_{2211} \otimes  \phi_{11})
		\end{array} \right.$
		& Cyclic
		\\

		$4$
		& $E_7(q)A_1(q)$
		& $\left\{ \begin{array}{c}
		(E_7(q).A_1(q), 10 \mbox{ chars}\otimes 1) \\
		(E_7(q).A_1(q), 10 \mbox{ chars} \otimes  \phi_{11})
		\end{array} \right.$
		& Trivial
		\\

		\hline \hline


		$5$
		& $E_6(q)A_2(q)$
		& $\left( \Phi_3^4, 1\right)$  
		& Not cyclic
		\\

		$6$
		& $E_6(q)A_2(q)$
		& $\left( \Phi_3^2 . {^3\!D_4}(q), {^3\!D_4}[-1] \right)$  
		& Not cyclic
		\\
		
		$7$
		& $E_6(q)A_2(q)$
		& $\left\{ \begin{array}{c}
		\left( \Phi_3 . {E}_6(q), \phi_{81,6}\right)\\
		\left( \Phi_3  . {E}_6(q), \phi_{81,10} \right) \\
		\left( \Phi_3 . {E}_6(q), \phi_{90,8} \right) 			
		\end{array} \right. $  
		& Cyclic
		\\
		
		\hline \hline
		
		
		
		$8$
		& $^2\!E_6(q){^2A}_2(q)$
		& $\left\{ \begin{array}{c}
		\left(\Phi_3^2 \Phi_6. {^2\!A_2(q)}, 1 \otimes 1 \right)  \\
		\left(\Phi_3^2 \Phi_6. {^2\!A_2(q)}, 1 \otimes \phi_{21} \right) \\
		\left(\Phi_3^2 \Phi_6. {^2\!A_2(q)}, 1 \otimes \phi_{111} \right)
		\end{array} \right. $
		& Not cyclic
		\\
		
		$9$
		& $^2E_6(q){^2A}_2(q)$
		& $\left\{ \begin{array}{c}
		\left({^2\!E_6(q)} .{^2\!A_2(q)}, 9 \mbox{ chars}  \otimes 1 \right)  \\
		\left({^2\!E_6(q)}. {^2\!A_2(q)},  9 \mbox{ chars}  \otimes \phi_{21} \right) \\
		\left({^2\!E_6(q)} .{^2\!A_2(q)},  9 \mbox{ chars}  \otimes \phi_{111} \right)
		\end{array} \right. $
		& Trivial
		\\
		
		\hline
		
	\end{longtable} 

\normalsize

Again,  all blocks   in ${\mathcal C}_b  $ correspond to the same  line of the table,  and  the corresponding  $\alpha $'s     have  the  same degree. Since the only relevant lines are those  with  non-cyclic entry in the last  column, $|\mathcal{C}_b| = 1$.
\end{proof}

 \section{Defining Characteristic and Ree and   Suzuki Groups.} \label{sec:reesuz} In this section $p$ denotes a prime number and $\bG$ a simple, simply connected group defined over $\overline \FF_p $. Let $F :\bG \to \bG $ be an endomorphism a power of which is a Frobenius morphism, allowing for the case that $\bG^F$ is a Ree or Suzuki group. We freely use the notation of Sections \ref{sec:alggroups} and \ref{sec:BDR} in this context.

 \begin{prop}
	\label{prop:defchar}  Suppose that $p=\ell $.
	 Let $b$ be a block of $\cO \textbf{G}^F$, $Z \leq Z(\bG^F)$ and let $\bar b$ be the block of $\cO(\bG^F/Z)$ dominated by $b$. Then $\cO (\bG^F/Z) \hat \sigma (\bar b)  \cong \cO (\bG^F / Z) \bar b$.
\end{prop}
\begin{proof}	
 If $Z(\bG^F) \leq C_2$, $Z(\bG^F) \cong C_2 \times C_2$ or if $b$ is a principal or Steinberg block, then as shown in \cite[Theorem 4.1]{F}, $\hat \sigma(b) = b$. 
	
Suppose that $Z(\textbf{G}^F) \cong C_m$ for some $m>2$ coprime to $\ell$ and assume that $b$ is not the principal or the Steinberg block. Let $\varphi = F_\ell^{\phi(m)-1}$ be the group automorphism of $\bG^F$ defined in \cite[Theorem 4.1]{F} with $F_\ell$ an $\FF_\ell$-split Steinberg endomorphism of $\bG$ and $\phi$ the Euler totient function. Let $\varphi'$ be the $\cO$-algebra  isomorphism induced by $\varphi$. Then, applying the arguments of \cite[Theorem 4.1]{F} to $\varphi'$ instead of to the $k$-algebra isomorphism induced by $\varphi$, it follows that $\varphi' (b) = \hat \sigma(b)$. The restriction $\varphi'|_{\cO \bG^F b} : \cO \bG^F b \rightarrow \cO \bG^F \hat \sigma(b) $ is also an $\cO$-algebra isomorphism, hence $\cO \bG^F b \cong \cO \bG^F \hat \sigma (b)$ as $\cO$-algebras. 

Now let $Z$ be a central subgroup of $\bG^F$. Then $Z$ is an $\ell'$-group by \cite[Table 24.2]{M/T} so by Lemma~\ref{lem:domblocks} (ii), $\cO \bG^F b \cong   \cO(\bG^F/Z) \bar b  $ and $\cO \bG^F \hat \sigma (b) \cong   \cO(\bG^F/Z) \overline{\hat{\sigma}(b)}  $. Since $\overline{\hat \sigma (b)} = \hat \sigma (\bar b)$ by Lemma~\ref{lem:domblocks} (iv), it follows that $\cO (\bG^F/Z)  \bar b \cong   \cO(\bG^F/Z) \hat \sigma(\bar b)  $. 
\end{proof}

 The next proposition deals with the Suzuki and Ree groups in  non-defining characteristic.
 \pagebreak 
  \begin{prop} \label{prop:suzukiree}     Let $ b$ be a block of $\cO \bG^F$.
  \begin{enumerate}[(i)]  \item   Suppose that $  \ell  \ne p =2$,  $\bG $ is  of type $B_2 $ and $ \bG^F  $ is  the  Suzuki group $\,^2B_2  (2^{2n+1} )  $.        Then $\cO \bG^F \hat \sigma  (b)  \cong \cO \bG^F b$.
\item  Suppose that  $  \ell \ne p =3 $, $\ell $,  $\bG $ is  of type $G_2 $ and $ \bG^F  $ is  the   Ree  group $\,^2 G_2  (3^{2n+1} )  $.        Then $\cO \bG^F \hat \sigma  (b)  \cong \cO \bG^F b$.
\item    Suppose  that  $  \ell  \ne  p=2$,  $\bG $ is  of type $F_4 $ and $ \bG^F  $ is  the   Ree  group $\,^2F_4  (2^{2n+1} )  $.     There exists an $F$-stable  Levi subgroup $\bL$ of $ \bG$ and a block $c$ of $\bL^F$ such that     $ \cO \bG^Fb$ is Morita equivalent to  $\cO \bL^F c $  and  $\cO \bL^F c \cong  \cO \bL^F \hat \sigma^r (c)$  for some $r \leq  2$.  
\end{enumerate}
\end{prop}
\begin{proof}  
If $b$   is the principal block, then $\hat \sigma (b) = b$ and there is nothing to prove.  Thus we may assume that $b$ is non-principal. First suppose that $\bG^F$ is  as in (i) or (ii).    By the proof of \cite[Theorem 5.7]{F}, $b$ has  either  cyclic or Klein four defect groups and we are done by  Proposition~\ref{prop:cyclicklein}.  

Assume from now on that $\bG^F $ is of type $\,^2F_4 (2^{2n+1})$. For $s$ a semisimple  element of $\ell'$-order of $\bG^{*F}$, $ \cE_{\ell}(\bG^F, s)$ is a union of $\ell $-blocks of $\bG^F$ (see \cite[p. 113]{H1}). Moreover, Theorem~\ref{thm: BDR} and hence Proposition \ref{prop:modcentralBDR} continue to hold with $ \bL^* $ and $\bL $ as in Section 5. We note that since $ Z(\bG)  =1 $, $N^F=\bL^F$ and $ Z =1 $.

Let $s \in \bG^{*F}$ be a semisimple $\ell'$-element such that $b \in \cE_{\ell} (\bG^F, s) $. If $ s =1 $, then again by the proof of \cite[Theorem 5.7]{F}, $b$ has either cyclic or Klein four defect groups  or $b$ is the principal block  and in these cases we are done as above (with $\bL =\bG$). So we may assume that $s\ne 1 $ and hence that $\bL$ is a proper Levi subgroup of $\bG$. In particular every simple component of $\bL$ and of $C_{\bG^*}(s) $  has classical type. If $C_{\bG^*}(s) $ contains no rational component of type $ \,^2B_2 $, then the conclusion of Lemma~\ref{lem:uniformprojections} and hence of Theorem \ref{prop:BDRresult} holds for $s$ and $b$. Further, since all  components of $\bL$ are of classical type, $a_{\bL^*}(s) \leq  2 $  whence  $r_{\bL^*}(s) =1 $. Thus, we obtain $\cO \bL^F c \cong  \cO \bL^F \hat \sigma (c)$. If $C_{\bG^*}(s) $ has a  rational component of type $ \,^2B_2 $, then $ C_{\bG^*}(s) =\bL^*$ and the conclusion of Theorem \ref{prop:central} holds. Moreover, $^{\hat \sigma ^2}\! \chi' = \chi'$ for all unipotent characters $\chi' \in \textnormal{Irr}(\bL^{F})$. Hence $\cO \bL^F c \cong  \cO \bL^F \hat \sigma ^2(c)$.
\end{proof}

\section{Sporadic groups, exceptional covering groups and  alternating   groups.  }  \label{sec:analogues}
We  give  further analogues of  results of  \cite{F}   over $\cO$. Let $b$ be a block of $\cO G$ for a finite group $G$. 

\begin{prop}\label{prop:sporadic}   Suppose that $G$ is a   covering group of   sporadic  simple group   or  the Tits  group.  Then  $\cO G\hat \sigma(b) \cong \cO Gb$.  \end{prop}

\begin{proof} If $b$ has cyclic defect,  we apply Proposition \ref{prop:cyclicklein}. For the remaining cases, 
the checks done in \cite{F} in GAP \cite{GAP4}  show that $\hat \sigma(b) = b$.
\end{proof}

 Next we treat the exceptional covering groups  of simple groups. The list  in \cite{F} was incomplete so we include a full proof over $\cO$ for the complete list here.

\begin{prop}\label{prop:exceptionalcovers}
	Let $G$ be one of the following exceptional covering groups 
	\begin{quote}
		$2.L_2(4)$, $2.L_3(2)$, $2.L_3(4)$, $4_1.L_3(4)$, $4_2.L_3(4)$, $6.L_3(4)$, $4^2.L_3(4)$, $12_1.L_3(4)$, $12_2.L_3(4)$,  $(4^2 \times 3).L_3(4)$, $2.L_4(2)$, $2.U_4(2)$, $2.U_6(2)$,  \break $6.U_6(2)$, $2^2.U_6(2)$, $(2^2 \times 3).U_6(2)$, $3.A_6$, $6.A_6$, $2.S_6(2)$, $2.Sz(8)$, $2^2.Sz(8)$, $2.O_8^+(2)$, $2^2.O_8^+(2)$, $2.G_2(4)$, $2.F_4(2)$, $2.{^2\!E}_6(2)$,  $6.{^2\!E_6}(2)$, $2^2.{^2\!E}_6(2)$, $(2^2 \times 3).{^2\!E}_6(2)$,  $3_1.U_4(3)$, $3_2.U_4(3)$, $6_1.U_4(3)$, $6_2.U_4(3)$, $3^2.U_4(3)$, $12_1.U_4(3)$, $12_2.U_4(3)$, $(3^2 \times 4).U_4(3)$, $3.O_7(3)$, $6.O_7(3)$, $3.G_2(3)$, $3.A_7$, $6.A_7$.
	\end{quote}
	Then  $\cO  G b    \cong  \cO G \hat \sigma (b)$.
\end{prop}

\begin{proof}  
	We may assume that $b$ has non-cyclic defect groups. By Lemma \ref{lem:domblocks} we may also assume that $Z(G)_{\ell'} $ is cyclic. 
	
	We see from checking in GAP \cite{GAP4} that in most cases the collections of $\ell$-blocks of $G$ with equal, non-cyclic defect and the same number and degrees of characters, none of which are rational valued, have size 1. Where there exists such a collection of size greater than 1, it has size 2. Suppose that $b_1$ and $b_2$ are non-principal $\ell$-blocks of $G$ with equal, non-cyclic defect and the same number and degrees of characters, none of which are rational valued. Then in all but 2 situations, it is possible to check directly in GAP \cite{GAP4} that there exists a single block $\hat{b}$ of a finite group $\hat{G}$ with $G \unlhd \hat{G}$ such that $\hat{b}$ covers both $b_1$ and $b_2$. Hence $\cO G b_1 \cong \cO G b_2$  via conjugation by an element of $\hat G$.  
	
	Finally, suppose that we are in one of the two remaining cases --  that is, either $\ell = 2$ and $G =  (4^2 \times 3).L_3(4)$, or $\ell = 3$ and $ G =(3^2 \times 4).U_4(3)$. Suppose that $b$ is a non-principal $\ell$-block of $ G$ with non-cyclic defect groups which contains no rational valued character and let $\bar b$ be the unique block of $G/Z(G)_{\ell}$ dominated by $b$. Then if $\ell=2  $ and $G =  (4^2 \times 3).L_3(4)$, $\bar b $ is one of two $2$-blocks  of $3.L_3(4)$ which are faithful  on  the  Sylow $3$-subgroup of  $G/Z(G)$. If $\ell=3  $ and $ G =  (3^2 \times 4).U_4(3) $ then  $\bar b $  is  one of two   $3$-blocks  of  $4.U_4(3) $ which are faithful  on  the   Sylow  $2$-subgroup of $G/Z(G)$. In both cases, by  a GAP  check  there exists a single block $\hat{b}$ of a finite group $\hat{G}$ with $G/Z(G)_{\ell}  \unlhd \hat{G}$ covering both the relevant blocks of $G/Z(G)_{\ell} $.  In other words, there exists  an automorphism of  $G/Z(G)_{\ell} $ sending $\bar b $ to   $\hat \sigma (\bar b) $.   Further in both cases, $G$   is  a universal covering group of $G/Z(G)_{\ell} $.   Since  every  automorphism of a  quasisimple finite group lifts to  an automorphism of its universal covering group (see Section 5.1  of  \cite{G/L/SIII}), it follows from Lemma  \ref{lem:domblocks} that $\cO G b \cong \cO G \hat \sigma (b)$. 
\end{proof}

\begin{prop}\label{prop:alt}
	Let $G$ be a quasi-simple finite group such that $G/Z(G)$ is an alternating group.  Then $\cO G  b \cong \cO G \hat \sigma ( b) $.
\end{prop}

\begin{proof}
	Let $\hat G$ denote a finite group such that $G \lhd \hat G$ and $\hat G / Z(\hat G)$ is a symmetric group, and let $\hat b$ denote a block of $\cO \hat G$ covering $b$. Since we may assume that  $b$  and hence  $\hat b$ does not have cyclic defect groups,  by Lemma~\ref{lem:domblocks} (i) and the proof of \cite[Theorem 3.1]{F}, $\hat \sigma (\hat b) = \hat b$. It follows that $b$ and $\hat \sigma(b)$ are both covered by $\hat b$ and therefore $\cO G b \cong \cO G \hat \sigma(b)$  via   conjugation by an element of $\hat G$. 
\end{proof}

\section{Proof of main theorems}\label{sec:proof}

Theorem~\ref{thm:newmaintheorem} is part of the following result.

\begin{thm}\label{thm:overallresult} 
	Let $G$ be a quasi-simple finite group and let $\bar{G} = G/Z(G)$. Let $b$ be a block of $\cO G$  and let $D$ be a defect group of $\cO Gb$. There exists  a finite group $N$ and a block $c$ of $\cO N$  such that   $ \cO Gb$ is Morita equivalent to  $\cO Nc $  and  $\cO Nc \cong  \cO N \hat \sigma^r (c)$  for some $r \leq  4 $.   Moreover,	
	\begin{enumerate}[(i)]
		\item if $\bar{G}$ is not a finite group of Lie type of type $E_8$ in characteristic $p \neq \ell$ then $r \leq  2 $, and 
		\item if $\bar{G} $ is a sporadic group, an alternating group, a finite group of Lie type in characteristic $\ell$ or a finite group of Lie type of type $A$, $B$ or $C$ in characteristic $p \neq \ell$, then $r = 1$.	
	\end{enumerate}
	Consequently, $\textit {mf }\!(kG\pi(b) ) \leq \textit{mf }\!(\cO Gb) \leq 4 $  and  $ \textit{sf }\!(\cO G b) \leq 4 |D|^2 ! $.    Unless $\bar{G}$ is a finite group of Lie type of type $E_8$ in characteristic $p \neq \ell$, $\textit{mf }\!(\cO Gb)  \leq 2$. If $\bar{G}$ is one of the groups in (ii), then  $\textit {mf }\!(\cO Gb)  =1 $.  
\end{thm}

\begin{proof}  
	We first consider the claim that there exists a finite group $N$ and a block $c$ of $\cO N$ such that $\cO Gb$ is Morita equivalent to $\cO N c$ and $\cO N c \cong \cO N \hat \sigma^r(c)$ with bounds for $r$ as in the statement. For $G$  an exceptional covering group  or  $\bar{G}$  an alternating group, a sporadic group, the Tits group, or a finite group of Lie type in characteristic $\ell$ (including the Suzuki and Ree groups), then the result holds with $N=G$ and $c=b$ by   Propositions~\ref{prop:exceptionalcovers}, \ref{prop:sporadic}, \ref{prop:alt}, and \ref{prop:defchar} respectively. Proposition~\ref{prop:suzukiree}  proves the   result  for $\bar G$ a  Suzuki  or Ree group   in characteristic $p \neq \ell$. Finally,   suppose that  $G$ is a    non-exceptional cover    of a finite group of Lie type  in characteristic $p \neq \ell$. Then  $G  =\bG^F/Z$   for $ \bG$  a simple simply-connected algebraic group   over  $ \overline \FF_p $, $ F :\bG \to \bG $ a Frobenius morphism and   $Z$ a central subgroup of $\bG^F$. Hence  the result holds by Propositions~\ref{prop:unipotent}, \ref{prop:ABC}, \ref{prop:D}, \ref{prop:G2F4E6}, \ref{prop:E7}, \ref{prop:E81}  and \ref{prop:E82}. Note that although we assume  in Section~\ref{sec:FGLT}  that $Z$ is an $\ell$-group, this is enough to show that the results hold in general by Lemma~\ref{lem:domblocks}.  This proves the first part of the theorem.

	Now $\textit{f }\!(\cO Nc) \leq r $. Hence,  
	\[ \textit{mf }\!(kG \pi(b) ) \leq \textit{mf }\!(\cO Gb)  = \textit{mf }\!(\cO Nc)  \leq \textit{f }\!(\cO Nc) \leq r, \]
	where the first   and second inequalities hold by    Proposition~\ref{prop:ELresults} (i), and the equality holds by  Proposition~\ref{prop:ELresults} (ii) and the fact that $\cO Gb$ and $\cO Nc$ are Morita equivalent. This proves the assertion about Morita Frobenius numbers.

	Finally, since   Morita equivalence between  blocks of finite group algebras preserves orders of defect groups,  the defect groups of  $\cO N c$  have order $|D|$, hence  
	\[\textit{sf }\!(\cO  Gb)     =  \textit{sf }\!(\cO  Nc)  \leq   \textit{f }\!(\cO Nc) |D|^2! \leq r |D|^2!  \leq 4   |D|^2!\]
	where the equality holds by Proposition~\ref{prop:ELresults}  (iii) and the first inequality holds by  Proposition~\ref{prop:ELresults}  (i). 
\end{proof}

Proof of Theorem~\ref{thm:speciallinear}:
\begin{proof}  
	The result follows from the first part of Theorem~\ref{thm:newmaintheorem}, \cite[Theorem 8.6]{H/K2} and    \cite[Theorem~1.4]{K}.
\end{proof}

Proof of Theorem~\ref{thm:sideshow}:
\begin{proof}
	This follows from the first part of Theorem~\ref{thm:newmaintheorem} and Lemma~\ref{lem:definedoverFl2}.
\end{proof}

\bibliographystyle{acm}
\bibliography{BIBLIOGRAPHY}

\end{document}